\documentclass[11pt,a4paper]{amsart}
\usepackage{amsmath, amsthm, amsfonts, amssymb}
\usepackage[bookmarks=false]{hyperref}

\newtheorem{letterthm}{Theorem}

\newtheorem{theo}{Theorem}[section]

\newtheorem{lem}[theo]{Lemma}
\newtheorem{prop}[theo]{Proposition}

\theoremstyle{definition}

\newtheorem{example}[theo]{Example}

\newtheorem{nota}[theo]{Notation}
\newtheorem{df}[theo]{Definition}
\newtheorem*{claim}{Claim}

\newcommand{\R}{\mathbf{R}}
\newcommand{\C}{\mathbf{C}}
\newcommand{\Z}{\mathbf{Z}}
\newcommand{\F}{\mathbf{F}}

\newcommand{\N}{\mathbf{N}}

\newcommand{\Ad}{\operatorname{Ad}}
\newcommand{\id}{\mathord{\text{\rm id}}}
\newcommand{\Tr}{\operatorname{Tr}}
\newcommand{\Aut}{\operatorname{Aut}}

\newcommand{\alg}{\text{\rm alg}}

\newcommand{\op}{\text{\rm op}}

\newcommand{\Ind}{\operatorname{Ind}}
\newcommand{\core}{\mathord{\text{\rm c}}}
\newcommand{\Proj}{\mathord{\text{\rm Proj}}_{\text{\rm f}}}
\newcommand{\Prob}{\mathord{\text{\rm Prob}}}
\newcommand{\weak}{\mathord{\text{\rm weak}}}
\newcommand{\Ball}{\mathord{\text{\rm Ball}}}

\newcommand{\rH}{\mathord{\text{\rm H}}}
\newcommand{\rL}{\mathord{\text{\rm L}}}

\newcommand{\cN}{\mathcal{N}}

\newcommand{\recht}{\rightarrow}

\newcommand{\cR}{\mathcal{R}}

\newcommand{\cK}{\mathcal{K}}

\newcommand{\cI}{\mathcal{I}}
\newcommand{\cM}{\mathcal{M}}
\newcommand{\cB}{\mathcal{B}}

\newcommand{\vphi}{\varphi}

\newcommand{\vphih}{\widehat{\varphi}}
\newcommand{\si}{\sigma}

\newcommand{\ot}{\otimes}
\newcommand{\ovt}{\mathbin{\overline{\otimes}}}

\newcommand{\rd}{\mathord{\text{\rm d}}}
\newcommand{\cA}{\mathcal{A}}

\newcommand{\cJ}{\mathcal{J}}
\newcommand{\dpr}{^{\prime\prime}}

\begin{document}

\title[Amalgamated free products with at most one Cartan subalgebra]{Amalgamated free product type III  factors with at most one Cartan subalgebra}

\author[R\'emi Boutonnet]{R\'emi Boutonnet}
\address{ENS Lyon \\
UMPA UMR 5669 \\
69364 Lyon cedex 7 \\
France}
\email{remi.boutonnet@ens-lyon.fr}
\thanks{R.B.\ and C.H.\ are supported by ANR grant NEUMANN}

\author[Cyril Houdayer]{Cyril Houdayer}
\address{CNRS-ENS Lyon \\
UMPA UMR 5669 \\
69364 Lyon cedex 7 \\
France}
\email{cyril.houdayer@ens-lyon.fr}

\author[Sven Raum]{Sven Raum}
\address{KU Leuven \\ Department of Mathematics \\ Celestijnenlaan 200B \\ B-3001 Leuven \\ Belgium}
\email{sven.raum@wis.kuleuven.be}
\thanks{S.R.\ is supported by K.U. Leuven BOF research grant OT/08/032.}


\begin{abstract}
We investigate Cartan subalgebras in nontracial amalgamated free product von Neumann algebras $M_1 \ast_B M_2$ over an amenable von Neumann subalgebra $B$. First, we settle the problem of the absence of Cartan subalgebra in arbitrary free product von Neumann algebras. Namely, we show that any nonamenable free product von Neumann algebra $(M_1, \varphi_1) \ast (M_2, \varphi_2)$ with respect to faithful normal states has no Cartan subalgebra. This generalizes the tracial case that was established in \cite{Io12a}. Next, we prove that any countable nonsingular ergodic equivalence relation $\mathcal R$ defined on a standard measure space and which splits as the free product $\mathcal R = \mathcal R_1 \ast \mathcal R_2$ of  recurrent subequivalence relations gives rise to a nonamenable factor $\rL(\mathcal R)$ with a unique Cartan subalgebra, up to unitary conjugacy. Finally, we prove unique Cartan decomposition for a class of group measure space factors $\rL^\infty(X) \rtimes \Gamma$ arising from nonsingular free ergodic actions $\Gamma \curvearrowright (X, \mu)$ on standard measure spaces of amalgamated groups $\Gamma = \Gamma_1 \ast_{\Sigma} \Gamma_2$ over a finite subgroup $\Sigma$.
\end{abstract}

\maketitle

\section{Introduction and main results}

A {\em Cartan subalgebra} $A$ in a von Neumann algebra $M$ is a unital maximal abelian $\ast$-subalgebra $A \subset M$ such that there exists a faithful normal conditional expectation $E_A : M \to A$ and such that the group of normalizing unitaries of $A$ inside $M$ defined by $\mathcal N_M(A) = \{u \in \mathcal U(M) : u A u^* = A\}$ generates $M$.

By a classical result of Feldman and Moore \cite{feldman-moore}, any Cartan subalgebra $A$ in a von Neumann algebra $M$ with separable predual arises from a countable nonsingular equivalence relation $\mathcal R$ on a standard measure space $(X, \mu)$ and a $2$-cocycle $\upsilon \in \rH^2(\mathcal R, \mathbf T)$. Namely, we have the following isomorphism of inclusions
$$(A \subset M) \; \cong \; (\rL^\infty(X) \subset \rL(\mathcal R, \upsilon)).$$
In particular, for any nonsingular free action $\Gamma \curvearrowright (X, \mu)$ of a countable discrete group $\Gamma$ on a standard measure space $(X, \mu)$, $\rL^\infty(X)$ is a Cartan subalgebra in the group measure space von Neumann algebra $\rL^\infty(X) \rtimes \Gamma$.

The presence of a Cartan subalgebra $A$ in a von Neumann algebra $M$ with separable predual is therefore an important feature which allows to divide the classification problem for $M$ up to $\ast$-isomorphism into two different questions: uniqueness of the Cartan subalgebra $A$ inside $M$ up to conjugacy and classification of the underlying countable nonsingular equivalence relation $\mathcal R$ up to orbit equivalence.

In \cite{connes-feldman-weiss}, Connes, Feldman and Weiss showed that any amenable countable nonsingular ergodic equivalence relation is hyperfinite and thus implemented by an ergodic $\Z$-action. This implies, together with \cite{Kr75}, that any two Cartan subalgebras inside an amenable factor are always conjugate by an automorphism.

The uniqueness of Cartan subalgebras up to conjugacy is no longer true in general for nonamenable factors. In \cite{CJ81}, Connes and Jones discovered the first examples of ${\rm II_1}$ factors with at least two Cartan subalgebras which are not conjugate by an automorphism. More concrete examples were later found by Popa and Ozawa in \cite{OP08}. We also refer to the recent work of Speelman and Vaes \cite{SV11} on ${\rm II_1}$ factors with uncountably many non (stably) conjugate Cartan subalgebras.

In the last decade, Popa's {\em deformation/rigidity theory} \cite{Po01, popa-malleable1} has led to a lot of progress in the classification of ${\rm II_1}$ factors arising from probability measure preserving (pmp) actions of countable discrete groups on standard probability spaces and from countable pmp equivalence relations. We refer to the recent surveys \cite{Po06b, Va10a, ioana-ecm} for an overview of this  topic.

We highlight below three breakthrough results regarding uniqueness of Cartan subalgebras in nonamenable ${\rm II_1}$ factors. In his pioneering article \cite{Po01}, Popa showed that any {\em rigid} Cartan subalgebra inside group measure space ${\rm II_1}$ factors $\rL^\infty(X) \rtimes \F_n$ arising from {\em rigid} pmp free ergodic actions $\F_n \curvearrowright (X, \mu)$ of the free group $\F_n$ ($n \geq 2$) is necessarily unitarily conjugate to $\rL^\infty(X)$. In \cite{OP07}, Ozawa and Popa proved that any {\em compact} pmp free ergodic action of the free group $\mathbf F_n$ ($n \geq 2$) gives rise to a ${\rm II_1}$ factor $\rL^\infty(X) \rtimes \F_n$ with unique Cartan decomposition, up to unitary conjugacy. This was the first result in the literature proving the uniqueness of Cartan subalgebras in nonamenable ${\rm II_1}$ factors. Recently, Popa and Vaes \cite{PV11} proved that {\em any} pmp free ergodic action of the free group $\F_n$ ($n \geq 2$) gives rise to a ${\rm II_1}$ factor $\rL^\infty(X) \rtimes \F_n$ with unique Cartan decomposition, up to unitary conjugacy. We refer to \cite{OP08, Ho09, CS11, CSU11, PV12, houdayer-vaes, Io12a} for further results in this direction.

Very recently, using \cite{PV11}, Ioana \cite{Io12a} obtained new results regarding the Cartan decomposition of {\em tracial} amalgamated free product von Neumann algebras $M_1 \ast_B M_2$. Let us highlight below two of Ioana's results \cite{Io12a}: any nonamenable tracial free product $M_1 \ast M_2$ has no Cartan subalgebra and any pmp free ergodic action $\Gamma \curvearrowright (X, \mu)$ of a free product group $\Gamma = \Gamma_1 \ast \Gamma_2$ with $|\Gamma_1| \geq 2$ and $|\Gamma_2| \geq 3$ gives rise to a ${\rm II_1}$ factor with  unique Cartan decomposition, up to unitary conjugacy.

In the present paper, we use Popa's deformation/rigidity theory to investigate Cartan subalgebras in {\em nontracial} amalgamated free product (AFP) von Neumann algebras $M_1 \ast_B M_2$ over an amenable von Neumann subalgebra $B$. We generalize some of Ioana's recent results \cite{Io12a} to this setting. The methods of proofs rely on a combination of results and techniques from \cite{PV11, houdayer-vaes, Io12a}.

\subsection*{Statement of the main results}

Using his  free probability theory, Voiculescu \cite{Vo95} proved that the free group factors $\rL(\F_n)$ ($n \geq 2$) have no Cartan subalgebra. This exhibited the first examples of ${\rm II_1}$ factors with no Cartan decomposition. This result was generalized later in \cite{Ju05} to free product ${\rm II_1}$ factors $M_1 \ast M_2$ of diffuse subalgebras which are embeddable into $R^\omega$. Finally, the general case of arbitrary tracial free product von Neumann algebras was recently obtained in \cite{Io12a} using Popa's deformation/rigidity theory.

The first examples of type ${\rm III}$ factors with no Cartan subalgebra were obtained in \cite{Sh00} as a consequence of \cite{Vo95}. Namely, it was shown that the unique free Araki-Woods factor of type ${\rm III}_\lambda$ ($0 < \lambda < 1$)  has no Cartan subalgebra. This result was vastly generalized later in \cite{houdayer-ricard} where it was proven that in fact any free Araki-Woods factor has no Cartan subalgebra.

Our first result settles the question of the absence of Cartan subalgebra in arbitrary free product von Neumann algebras.

\begin{letterthm}\label{thmA}
Let $(M_1, \varphi_1)$ and $(M_2, \varphi_2)$ be any von Neumann algebras with separable predual endowed with faithful normal states such that $\dim M_1 \geq 2$ and $\dim M_2 \geq 3$. Then the free product von Neumann algebra $(M, \varphi) = (M_1, \varphi_1) \ast (M_2, \varphi_2)$ has no Cartan subalgebra.  
\end{letterthm}

Observe that when $\dim M_1 = \dim M_2 = 2$, the free product $M = M_1 \ast M_2$ is hyperfinite by \cite[Theorem 1.1]{Dy92} and so has a Cartan subalgebra. Note that the questions of factoriality, type classification and fullness for arbitrary free product von Neumann algebras were recently 
settled in \cite{ueda-advances}. These results are used in the proof of Theorem \ref{thmA}.

We next investigate more generally Cartan subalgebras in nontracial AFP von Neumann algebras $M = M_1 \ast_B M_2$ over an amenable von Neumann subalgebra $B$. Even though we do not get a complete solution in that setting, our second result shows that, under fairly general assumptions, any Cartan subalgebra $A \subset M$ can be embedded into $B$ inside $M$, in the sense of Popa's intertwining techniques. We refer to Section \ref{preliminairies} for more information on these intertwining techniques and the notation $A \preceq_M B$. Recall from \cite[Definition 5.1]{houdayer-vaes} that an inclusion of von Neumann algebras $P \subset M$ has no {\em trivial corner} if for all nonzero projections $p \in P' \cap M$, we have $Pp \neq pMp$.

\begin{letterthm}\label{thmB}
For $i \in \{ 1, 2 \}$, let $B \subset M_i$ be any inclusion of von Neumann algebras with separable predual and with faithful normal conditional expectation $E_i : M_i \to B$. Let $(M, E) = (M_1, E_1) \ast_B (M_2, E_2)$ be the corresponding amalgamated free product von Neumann algebra. Assume that $B$ is a finite amenable von Neumann algebra.

Assume moreover that:
\begin{itemize}
\item Either both $M_1$ and $M_2$ have no amenable direct summand.
\item Or $B$ is of finite type ${\rm I}$, $M_1$ has no amenable direct summand and the inclusion $B \subset M_2$ has no trivial corner.
\end{itemize}
If $A \subset M$ is a Cartan subalgebra, then $A \preceq_M B$.
\end{letterthm}
A similar result was obtained for tracial AFP von Neumann algebras in \cite[Theorem 1.3]{Io12a}.

The first examples of type ${\rm III}$ factors with unique Cartan decomposition were recently obtained in \cite{houdayer-vaes}. Namely, it was shown that any nonamenable nonsingular free ergodic action $\Gamma \curvearrowright (X, \mu)$ of a Gromov hyperbolic group on a standard measure space gives rise to a factor $\rL^\infty(X) \rtimes \Gamma$ with unique Cartan decomposition, up to unitary conjugacy. This generalized the probability measure preserving case that was established in \cite{PV12}.

In order to state our next results, we need to introduce some terminology. Let $\mathcal R$ be a countable nonsingular equivalence relation on a standard measure space $(X, \mu)$ and denote by $\rL(\mathcal R)$ the von Neumann algebra of the equivalence relation $\mathcal R$ (\cite{feldman-moore}). Following \cite[Definition 2.1]{adams}, we say that $\mathcal R$ is {\em recurrent} if for all measurable subsets $\mathcal U \subset X$ such that $\mu(\mathcal U) > 0$, the set $[x]_{\mathcal R} \cap \mathcal U$ is infinite for almost every $x \in \mathcal U$. This is equivalent to saying that $\rL(\cR)$ has no type ${\rm I}$ direct summand. We then say that a nonsingular action $\Gamma \curvearrowright (X, \mu)$ of a countable discrete group on a standard measure space  is {\em recurrent} if the corresponding orbit equivalence relation $\mathcal R(\Gamma \curvearrowright X)$ is recurrent.

Our next result provides a new class of type ${\rm III}$ factors with unique Cartan decomposition, up to unitary conjugacy. These factors arise from countable nonsingular ergodic equivalence relations $\mathcal R$ which split as a free product $\mathcal R = \mathcal R_1 \ast \mathcal R_2$ of arbitrary recurrent subequivalence relations. We refer to \cite[Definition IV.6]{Ga99} for the notion of {\em free product} of countable nonsingular equivalence relations.

\begin{letterthm}\label{thmC}
Let $\mathcal R$ be any countable nonsingular ergodic equivalence relation on a standard measure space $(X, \mu)$ which splits as a free product $\mathcal R = \mathcal R_1 \ast \mathcal R_2$ such that the subequivalence relation $\mathcal R_i$ is recurrent for all $i \in \{1, 2\}$.

Then the nonamenable factor $\rL(\mathcal R)$ has $\rL^\infty(X)$ as its unique Cartan subalgebra, up to unitary conjugacy. In particular, for any nonsingular ergodic equivalence relation $\mathcal S$ on a standard measure space $(Y, \eta)$ such that $\rL(\mathcal R) \cong \rL(\mathcal S)$, we have $\mathcal R \cong \mathcal S$.
\end{letterthm}

Observe that Theorem \ref{thmC} generalizes \cite[Corollary 1.4]{Io12a} where the same result was obtained for countable pmp equivalence relations under additional assumptions. Note that in the case when $\mathcal R_1$ is {\em nowhere amenable}, that is, $\rL(\mathcal R_1)$ has no amenable direct summand and $\mathcal R_2$ is recurrent, Theorem \ref{thmC} is a consequence of Theorem \ref{thmB} and \cite[Theorem 2.5]{houdayer-vaes}. However, Theorem \ref{thmB} does not cover the case when both $\mathcal R_1$ and $\mathcal R_2$ are amenable. So, in the setting of von Neumann algebras arising from countable nonsingular equivalence relations, Theorem \ref{thmC} is a generalization of Theorem \ref{thmB} in the sense that we are able to remove the nonamenability assumption on $M_1 = \rL(\mathcal R_1)$.

Finally, when dealing with certain nonsingular free ergodic actions $\Gamma \curvearrowright (X, \mu)$ of amalgamated groups $\Gamma_1 \ast_\Sigma \Gamma_2$, we obtain new examples of group measure space type ${\rm III}$ factors with unique Cartan decomposition, up to unitary conjugacy.

\begin{letterthm}\label{thmD}
Let $\Gamma = \Gamma_1 \ast_\Sigma \Gamma_2$ be any amalgamated free product of countable discrete groups such that $\Sigma$ is finite and $\Gamma_i$ is infinite for all $i \in \{ 1, 2 \}$. Let $\Gamma \curvearrowright (X, \mu)$ be any nonsingular free ergodic action on a standard measure space such that for all $i \in \{ 1, 2 \}$, the restricted action $\Gamma_i \curvearrowright (X, \mu)$ is recurrent.
 
Then the group measure space factor $\rL^\infty(X) \rtimes \Gamma$ has $\rL^\infty(X)$ as its unique Cartan subalgebra, up to unitary conjugacy. 
\end{letterthm}

Observe that Theorem \ref{thmD} generalizes the probability measure preserving case that was established in \cite[Theorem 1.1]{Io12a}.

In the spirit of \cite[Corollary B]{houdayer-vaes}, we obtain the following interesting consequence. Let $\Gamma = \Gamma_1 \ast \Gamma_2$ be an arbitrary free product group such that $\Gamma_1$ is amenable and infinite and $|\Gamma_2| \geq 2$. Then we get group measure space factors of the form $\rL^\infty(X) \rtimes \Gamma$ with unique Cartan decomposition, having any possible type and with any possible flow of weights in the type ${\rm III_0}$ case.

We finally mention that, unlike the probability measure preserving case \cite[Theorem 1.1]{Io12a},  the assumption of recurrence of the action $\Gamma_i \curvearrowright (X, \mu)$ for all $i \in \{1, 2\}$ is necessary. Indeed, using \cite{SV11}, we exhibit in Section \ref{non-uniqueness} a class of nonamenable infinite measure preserving free ergodic actions $\Gamma \curvearrowright (X, \mu)$ of free product groups $\Gamma = \Gamma_1 \ast \Gamma_2$ such that the corresponding type ${\rm II_\infty}$ group measure space factor $\rL^\infty(X) \rtimes \Gamma$ has uncountably many non conjugate Cartan subalgebras.

\subsection*{Comments on the proofs}
As we already mentioned above, the proofs of our main results rely heavily on results and techniques from \cite{PV11, houdayer-vaes, Io12a}. Let us describe below the main three ingredients which are needed. We will mainly focus on the proof of Theorem \ref{thmA}.

Denote by $(M, \varphi) = (M_1, \varphi_1) \ast (M_2, \varphi_2)$ an arbitrary free product of von Neumann algebras as in Theorem \ref{thmA}. For simplicity, we may assume that $M$ is a factor. In the case when both $M_1$ and $M_2$ are amenable, $M$ is already known to have no Cartan subalgebra by \cite[Theorem 5.5]{houdayer-ricard}. So we may assume that $M_1$ is not amenable. Using \cite{Dy92, ueda-advances}, we may further assume that $M_1$ has no amenable direct summand and $M_2 \neq \C$. By contradiction, assume that $A \subset M$ is a Cartan subalgebra.

We first use Connes-Takesaki's {\em noncommutative flow of weights} \cite{connes73, connestak, Ta03} in order to work inside the semifinite von Neumann algebra $\core(M)$ which is the {\em continuous core} of $M$. We obtain a canonical decomposition of $\core(M)$ as the semifinite amalgamated free product von Neumann algebra $\core(M) = \core(M_1) \ast_{\rL(\R)} \core(M_2)$. Moreover $\core(A) \subset \core(M)$ is a Cartan subalgebra.

Next, we use Popa's intertwining techniques in the setting of nontracial von Neumann algebras that were developed in \cite[Section 2]{houdayer-vaes}. Since $A$ is diffuse, we show that necessarily $\core(A) \npreceq_{\core(M)} \rL(\R)$ (see Proposition \ref{intertwining-core}).

Finally, we extend Ioana's techniques from \cite[Sections 3,4]{Io12a} to {\em semifinite} AFP von Neumann algebras (see Theorems \ref{dichotomy-intertwining} and \ref{dichotomy-amenability}). The major difference though between our approach and Ioana's approach is that we cannot use the spectral gap techniques from \cite[Section 5]{Io12a}. The main reason why Ioana's approach cannot work here is that $\core(M)$ is not full in general even though $M$ is a full factor. Instead, we strengthen \cite[Theorem 4.1]{Io12a} in the following way. We show that the presence of the Cartan subalgebra $\core(A) \subset \core(M)$ which satisfies $\core(A) \npreceq_{\core(M)} \rL(\R)$ forces {\em both} $\core(M_1)$ and $\core(M_2)$ to have an amenable direct summand. Therefore, both $M_1$ and $M_2$ have an amenable direct summand as well. Since we assumed that $M_1$ had no amenable direct summand, this is a contradiction.

\subsection*{Acknowledgments} We are grateful to Adrian Ioana and Stefaan Vaes for their useful comments regarding a first draft of this paper. We also thank the referee for carefully reading the paper.

\tableofcontents

\section{Preliminaries}\label{preliminairies}

Since we want the paper to be as self contained as possible, we recall in this section all the necessary background that will be needed for the proofs of the main results.

\subsection{Intertwining techniques}

All the von Neumann algebras that we consider in this paper are always assumed to be $\sigma$-finite. Let $M$ be a von Neumann algebra. We say that a von Neumann subalgebra $P \subset 1_P M1_P$ is with {\em expectation} if there exists a faithful normal conditional expectation $E_P : 1_P M 1_P \to P$. Whenever $\mathcal V \subset M$ is a linear subspace, we denote by $\Ball(\mathcal V)$ the {\em unit ball} of $\mathcal V$ with respect to the uniform norm $\|\cdot\|_\infty$. We will sometimes say that a von Neumann algebra $(M, \tau)$ is {\em tracial} if $M$ is endowed with a faithful normal tracial state $\tau$.

In \cite{Po01, popa-malleable1}, Popa discovered the following powerful method to unitarily conjugate subalgebras of a finite von Neumann algebra. Let $M$ be a finite von Neumann algebra and $A\subset 1_A M 1_A$, $B \subset 1_B M 1_B$ von Neumann subalgebras. By \cite[Corollary 2.3]{popa-malleable1} and \cite[Theorem A.1]{Po01}, the following statements are equivalent:
\begin{itemize}
\item There exist projections $p \in A$ and $q \in B$, a nonzero partial isometry $v \in p M q$ and a unital normal $\ast$-homomorphism $\varphi : p A p \to q B q$ such that $a v = v \varphi(a)$ for all $a \in A$.

\item There exist $n \geq 1$, a possibly nonunital normal $\ast$-homomorphism $\pi : A \to \mathbf M_n(B)$ and a nonzero partial isometry $v \in \mathbf M_{1, n}(1_A M 1_B)$ such that $a v = v \pi(a)$ for all $a \in A$.

\item There is no net of unitaries $(w_k)$ in $\mathcal U(A)$ such that $E_B(x^* w_k y) \to 0$ $\ast$-strongly for all $x, y \in 1_A M 1_B$.
\end{itemize}
If one of the previous equivalent conditions is satisfied, we say that $A$ {\it embeds into} $B$ {\it inside} $M$ and write $A \preceq_M B$.

We will need the following generalization of Popa's Intertwining Theorem which was proven in \cite[Theorems 2.3, 2.5]{houdayer-vaes}. A further generalization can also be found in \cite[Proposition 3.1]{ueda12}.

\begin{theo}\label{intertwining-general}
Let $M$ be any von Neumann algebra. Let $A \subset 1_A M 1_A$ and $B \subset 1_B M 1_B$ be von Neumann subalgebras such that $B$ is finite and with expectation $E_B : 1_B M 1_B \to B$. The following are equivalent.
\begin{enumerate}
\item There exist $n \geq 1$, a possibly nonunital normal $\ast$-homomorphism $\pi : A \to \mathbf M_n(B)$ and a nonzero partial isometry $v \in \mathbf M_{1, n}(1_A M 1_B)$ such that $a v = v \pi(a)$ for all $a \in A$.

\item There is no net of unitaries $(w_k)$ in $\mathcal U(A)$ such that $E_B(x^* w_k y) \to 0$ $\ast$-strongly for all $x, y \in 1_A M 1_B$.
\end{enumerate}
Moreover, when $M$ is a factor and $A, B \subset M$ are both Cartan subalgebras, the previous conditions are equivalent with the following:
\begin{enumerate}
\item [(3)] There exists a unitary $u \in \mathcal U(A)$ such that $u A u^* = B$.
\end{enumerate}
\end{theo}

\begin{df}\label{intertwining-general-definition}
Let $M$ be any von Neumann algebra. Let $A \subset 1_A M1_A$ and $B \subset 1_B M 1_B$ be von Neumann subalgebras such that $B$ is finite and with expectation. We say that $A$ \emph{embeds into} $B$ \emph{inside} $M$ and denote $A \preceq_M B$ if one of the equivalent conditions of Theorem \ref{intertwining-general} is satisfied.
\end{df}

Observe that when $1_A$ and $1_B$ are finite projections in $M$ then $1_A \vee 1_B$ is finite, and $A \preceq_M B$ in the sense of Definition \ref{intertwining-general-definition} if and only if $A \preceq_{(1_A \vee 1_B) M (1_A \vee 1_B)} B$ holds in the usual sense for finite von Neumann algebras.

In case of semifinite von Neumann algebras, we recall the following useful intertwining result (see \cite[Lemma 2.2]{houdayer-ricard}). When $(\mathcal B, \Tr)$ is a semifinite von Neumann algebra endowed with a semifinite faithful normal trace, we will denote by $\Proj(\mathcal B)$ the set of all nonzero finite trace projections of $\mathcal B$. We will denote by $\|\cdot\|_{2, \Tr}$ the $\rL^2$-norm associated with the trace $\Tr$. 

\begin{lem}\label{intertwining-general-HR}
Let $(\mathcal M, \Tr)$ be a semifinite von Neumann algebra endowed with a semifinite faithful normal trace. Let $\mathcal B \subset \mathcal M$ be a von Neumann subalgebra such that $\Tr | \mathcal B$ is semifinite. Denote by $E_{\mathcal B} : \mathcal M \to \mathcal B$ the unique trace-preserving faithful normal conditional expectation.

Let $p \in \Proj(\mathcal M)$ and $\mathcal A \subset p \mathcal M p$ any von Neumann subalgebra. The following conditions are equivalent:
\begin{enumerate}
\item For every $q \in \Proj(\mathcal B)$, we have $\mathcal A \npreceq_{\mathcal M} q \mathcal B q$.
\item There exists an increasing sequence of projections $q_n \in \Proj(\mathcal B)$ such that $q_n \to 1$ strongly and $\mathcal A \npreceq_{\mathcal M} q_n \mathcal B q_n$ for all $n \in \N$.
\item There exists a net of unitaries $w_k \in \mathcal U(\mathcal A)$ such that $\lim_k \|E_{\mathcal B}(x^* w_k y)\|_{2, \Tr} = 0$ for all $x, y \in p \mathcal M$.
\end{enumerate}
\end{lem}

\begin{proof}
$(1) \Rightarrow (2)$ is obvious.

$(2) \Rightarrow (3)$ Let $\mathcal F \subset \Ball( p \mathcal M)$ be a finite subset and $\varepsilon > 0$. We need to show that there exists $w \in \mathcal U(\mathcal A)$ such that $\|E_{\mathcal B}(x^* w y)\|_{2, \Tr} < \varepsilon$ for all $x, y \in \mathcal F$. Since the projection $p$ has finite trace, there exists $n \in \N$ large enough such that 
$$\|q_nx^*p - x^*p\|_{2, \Tr} + \|p y q_n - p y\|_{2, \Tr} < \frac{\varepsilon}{2}, \forall x, y \in \mathcal F.$$
Put $q = q_n$. Since $\mathcal A \npreceq_{\mathcal M} q \mathcal B q$, there exists a net $w_k \in \mathcal U(\mathcal A)$ such that $\lim_k \|E_{q \mathcal B q}(a^* w_k b)\|_{2, \Tr} = 0$ for all $a, b \in p \mathcal M q$. Appying this to $a = p x q$ and $b = p y q$, if we take $w = w_k$ for $k$ large enough, we get $\|E_{\mathcal B}(q x^*p \, w \, p y q)\|_{2, \Tr} = \| E_{q\mathcal B q}(q x^*p \, w \, p y q)\|_{2, \Tr} < \frac{\varepsilon}{2}$. Therefore, $\|E_{\mathcal B} (x^* w y)\|_{2, \Tr} < \varepsilon$.

$(3) \Rightarrow (1)$ Let $q \in \Proj(\mathcal B)$ and put $e = p \vee q$. Let $\lambda = \Tr(e) < \infty$ and denote by $\|\cdot\|_2$ the $\rL^2$-norm with respect to the normalized trace on $e \mathcal M e$. For all $x, y \in p \mathcal M q$, we have 
$$\lim_k \|E_{q \mathcal B q} (x^* w_k y)\|_2 = \lambda^{-1/2} \lim_k \|E_{q \mathcal B q}(x^* w_k y)\|_{2, \Tr} = 0.$$
This means exactly that $\mathcal A \npreceq_{e \mathcal M e} q\mathcal B q$ in the usual sense for tracial von Neumann algebras and so $\mathcal A \npreceq_{ \mathcal M} q\mathcal B q$.
\end{proof}

Let $\Gamma$ be any countable discrete group and $\mathcal S$ any nonempty collection of subgroups of $\Gamma$. Following \cite[Definition 15.1.1]{BO08}, we say that a subset $\mathcal F \subset \Gamma$ is {\em small relative to} $\mathcal S$ if there exist $n \geq 1$, $\Sigma_1, \dots, \Sigma_n \in \mathcal S$ and $g_1, h_1, \dots, g_n, h_n \in \Gamma$ such that $\mathcal F \subset \bigcup_{i = 1}^n g_i \Sigma_i h_i$.

We will need the following generalization of \cite[Proposition 2.6]{vaes-cohomology} and \cite[Lemma 2.7]{houdayer-vaes}.

\begin{prop}\label{intertwining-semifinite}
Let $(\mathcal B, \Tr)$ be a semifinite von Neumann algebra endowed with a semifinite faithful normal trace. Let $\Gamma \curvearrowright (\mathcal B, \Tr)$ be a trace preserving action of a countable discrete group $\Gamma$ on $(\mathcal B, \Tr)$ and denote by $\mathcal M = \mathcal B \rtimes \Gamma$ the corresponding semifinite crossed product von Neumann algebra. Let $p \in \Proj(\mathcal M)$ and $\mathcal A \subset p \mathcal M p$ any von Neumann subalgebra. Denote $\mathcal P = \mathcal N_{p\mathcal Mp}(\mathcal A)\dpr$. 

For every subset $\mathcal F \subset \Gamma$ which is small relative to $\mathcal S$, denote by $P_{\mathcal F}$ the orhogonal projection from $\rL^2(\mathcal M, \Tr)$ onto the closed linear span of $\{xu_g : x \in \mathcal B \cap \rL^2(\mathcal B, \Tr), g \in \mathcal F\}$. 

\begin{enumerate}
\item The set $\mathcal J = \{e \in \mathcal A' \cap p\mathcal Mp : \mathcal A e \npreceq_{ \mathcal M} q (\mathcal B \rtimes \Sigma) q, \forall \Sigma \in \mathcal S, \forall q \in \Proj(\mathcal B)\}$ is directed and attains its maximum in a projection $z$ which belongs to $\mathcal Z(\mathcal P)$.
\item There exists a net $(w_k)$ in $\mathcal U(\mathcal Az)$ such that $\lim_k \|P_{\mathcal F}(w_k)\|_{2, \Tr} = 0$ for every subset $\mathcal F \subset \Gamma$ which is small relative to $\mathcal S$.
\item For every $\varepsilon > 0$, there exists a subset $\mathcal F \subset \Gamma$ which is small relative to $\mathcal S$ such that $\|a - P_{\mathcal F}(a)\|_{2, \Tr} < \varepsilon$ for all $a \in \mathcal A(p - z)$.
\end{enumerate}
\end{prop}

\begin{proof}
$(1)$ In order to show that the set $\mathcal J$ is directed and attains its maximum, it suffices to prove that 
whenever $(e_i)_{i \in I}$ is a family of projections in $\mathcal A' \cap p\mathcal Mp$ and $e = \bigvee_{i \in I} e_i$, if $e \notin \mathcal J$, then there exists $i \in I$ such that $e_i \notin \mathcal J$. If $e \notin \mathcal J$, there exist $\Sigma \in \mathcal S$ and $q \in \Proj(\mathcal B)$ such that $\mathcal A e \preceq_{ \mathcal M} q (\mathcal B \rtimes \Sigma) q$. Let $n \geq 1$, a nonzero partial isometry $v \in \mathbf M_{1, n}(\C) \otimes e \mathcal M q$ and a normal $\ast$-homomorphism $\varphi : \mathcal A e \to \mathbf M_n(q(\mathcal B \rtimes \Sigma)q)$ such that $a v = v \varphi(a)$ for all $a \in \mathcal A e$. By definition we have $e v = v$. Choose $i \in I$ such that $e_i v \neq 0$ and denote by $w \in \mathbf M_{1, n}(\C) \otimes e_i \mathcal M q$ the polar part of $e_i v$. Since $a w = w \varphi(a)$ for all $a \in \mathcal A e$, it follows that $\mathcal A e_i \preceq_{ \mathcal M } q (\mathcal B \rtimes \Sigma) q$. Hence, $e_i \notin \mathcal J$.

Denote by $z$ the maximum of the set $\mathcal J$. It is easy to see that $u z u^* \in \mathcal J$ whenever $u \in \mathcal N_{p\mathcal Mp}(\mathcal A)$, hence $u z u^* = z$. Therefore $z \in \mathcal Z(\mathcal P)$.

$(2)$ We have that $\mathcal A z \npreceq_{\mathcal M } q(\mathcal B \rtimes \Sigma)q$ for all $\Sigma \in \mathcal S$ and all $q \in \Proj(\mathcal B)$. Let $\varepsilon > 0$ and $\mathcal F \subset \Gamma$ a subset which is small relative to $\mathcal S$. We show that we can find $w \in \mathcal U(\mathcal A z)$ such that $\|P_{\mathcal F}(w)\|_{2, \Tr} < \varepsilon$.

Let $\mathcal F \subset \bigcup_{i = 1}^n g_i \Sigma_i h_i$ with $\Sigma_1, \dots, \Sigma_n \in \mathcal S$ and $g_1, h_1, \dots, g_n, h_n \in \Gamma$. Consider the semifinite von Neumann algebra $\mathbf M_n(\mathcal M)$ together with the diagonal subalgebra $\mathcal Q = \bigoplus_{i = 1}^n \mathcal B \rtimes \Sigma_i$. Observe that the canonical trace on $\mathbf M_n(\mathcal M)$ is still semifinite on $\mathcal Q$. Consider moreover the trace preserving $\ast$-embedding $\rho : \mathcal M \to \mathbf M_n(\mathcal M) : x \mapsto x \oplus \cdots \oplus x$. 

Since $\mathcal A z \npreceq_{ \mathcal M} q(\mathcal B \rtimes \Sigma_i)q$ for all $i \in \{1, \dots, n\}$ and all $q \in \Proj(\mathcal B)$, we get that $\rho(\mathcal A z) \npreceq_{ \mathbf M_n(\mathcal M)} \rho(q) \mathcal Q \rho(q)$ for all $q \in \Proj(\mathcal B)$ by the first criterion in Lemma \ref{intertwining-general-HR}. Then by the second criterion in Lemma \ref{intertwining-general-HR}, there exists a net $w_k \in \mathcal U(\mathcal A z)$ such that 
$$\lim_k \|E_{\mathcal B \rtimes \Sigma_i}(x w_k y)\|_{2, \Tr} = 0, \forall x, y \in \mathcal M, \forall i \in \{1, \dots, n\}.$$
Recall that $P_{g \Sigma h}(x) = u_g E_{\mathcal B \rtimes \Sigma}(u_g^* x u_h^*) u_h$ for all $x \in \mathcal M \cap \rL^2(\mathcal M, \Tr)$. Applying what we have just proved to $x = u_{g_i}^*$ and $y = u_{h_i}^*$, we get that $\lim_k \|P_{g_i \Sigma_i h_i}(w_k)\|_{2, \Tr} = 0$ for all $i \in \{1, \dots, n\}$. Therefore $\lim_k \|P_{\mathcal F}(w_k)\|_{2, \Tr} = 0$.

$(3)$ By construction, for any projection $e \leq p - z$, there exist $\Sigma \in \mathcal S$ and $q \in \Proj(\mathcal B)$ such that $\mathcal Ae \preceq_{\mathcal M} q(\mathcal B \rtimes \Sigma)q$. Let $\varepsilon > 0$. Choose $\ell \geq 1$ and $e_1, \dots, e_\ell \in \mathcal A' \cap p \mathcal M p$ pairwise orthogonal projections such that:
\begin{itemize}
\item For every $i \in \{1, \dots, \ell\}$, $e_i \leq p - z$ and $e = e_1 + \cdots + e_\ell$ satisfies $\|(p - z) - e\|_{2, \Tr} \leq \varepsilon / 3$.
\item For every $i \in \{1, \dots, \ell\}$, there exist $n_i \geq 1$, $\Sigma_i \in \mathcal S$, a projection $q_i \in \Proj(\mathcal B)$, a nonzero partial isometry $v_i \in \mathbf M_{1, n_i}(\C) \otimes e_i \mathcal M q_i$ and a normal $\ast$-homomorphism $\varphi_i : \mathcal A \to \mathbf M_{n_i}(q_i (\mathcal B \rtimes \Sigma_i) q_i)$ such that $v_i v_i^* = e_i$ and $a v_i = v_i \varphi_i(a)$ for all $a \in \mathcal A$.
\end{itemize}

Put $n = n_1 + \cdots + n_\ell$, $q = \bigvee_{i = 1}^\ell q_i$ and define $\varphi : \mathcal A \to \bigoplus_{i = 1}^\ell q_i (\mathcal B \rtimes \Sigma_i) q_i \subset \mathbf M_n(q \mathcal M q)$ by putting together the $\varphi_i$ diagonally. Similarly, define the partial isometry $v \in \mathbf M_{1, n}(\C) \otimes e \mathcal M q$ such that $vv^* = e$ and $a v = v \varphi(a)$ for all $a \in \mathcal A$. 

Using Kaplansky density theorem, choose $v_0 \in \mathbf M_{1, n}(\C) \otimes q(\mathcal B \rtimes_{\alg} \Gamma)q$ such that $\|v_0\|_\infty \leq 1$ and $\|v - v_0\|_{2, \Tr} < \varepsilon/3$. Define $\mathcal G \subset \Gamma$ the finite subset such that $v_0$ belongs to the linear span of $\{e_{1i} \otimes exu_gq : x \in \mathcal B, g \in \mathcal G, 1 \leq i \leq \ell\}$. Put $\mathcal F = \bigcup_{i = 1}^\ell \bigcup_{g, h \in \mathcal G} g \Sigma_i h^{-1}$.

Let $a \in \Ball(\mathcal A(p - z))$ and write $a = a(p - z - e) + a e$. Observe that $\|a(p - z - e)\|_{2, \Tr} \leq \|a\|_\infty \|p - z - e\|_{2, \Tr} < \varepsilon/3$. Since $a e = v \varphi(a) v^*$, it follows that $a e$ lies at a distance less than $2 \varepsilon/3$ from $v_0 \varphi(a) v_0^*$. Observe that by construction $P_{\mathcal F}(v_0 \varphi(a) v_0^*) = v_0 \varphi(a) v_0^*$. Therefore, $a$ lies at a distance less than $\varepsilon$ from the range of $P_{\mathcal F}$.
\end{proof}

\subsection{Amalgamated free product von Neumann algebras}

For $i \in \{1, 2\}$, let $B \subset M_i$ be an inclusion of von Neumann algebras with expectation $E_i : M_i \to B$. Recall that the \emph{amalgamated free product} $(M, E) = (M_1, E_1) \ast_ B (M_2, E_2)$ is the von Neumann algebra $M$ generated by $M_1$ and $M_2$ where the faithful normal conditional expectation $E : M \to B$ satisfies the freeness condition:
$$E(x_1 \cdots x_n) = 0 \mbox{ whenever } x_j \in M_{i_j} \ominus B \mbox{ and } i_j \neq i_{j + 1} \; .$$
Here and in what follows, we denote by $M_i \ominus B$ the kernel of the conditional expectation $E_i : M_i \to B$. We refer to \cite{voiculescu85,voiculescu92,ueda-pacific} for more details on the construction of amalgamated free products in the framework of von Neumann algebras.

Assume that $\Tr$ is a semifinite faithful normal trace on $B$ such that for all $i \in \{1, 2\}$, the weight $\Tr \circ E_i$ is a trace on $M_i$. Then the weight $\Tr \circ E$ is a trace on $M$ by \cite[Theorem 2.6]{ueda-pacific}. In that case, we will say that the amalgamated free product $(M, E) = (M_1, E_1) \ast_{B} (M_2, E_2)$ is {\em semifinite}. Whenever we consider a semifinite faithful normal trace on a semifinite amalgamated free product $(M, E) = (M_1, E_1) \ast_{B} (M_2, E_2)$, we will always assume that $\Tr \circ E = \Tr$ and $\Tr | B$ is semifinite.

The following proposition is a semifinite analogue of \cite[Theorem 1.1]{IPP05}. The proof of Theorem \ref{controlling-afp} is essentially contained in \cite[Theorem 2.4]{CH08}.

\begin{theo}\label{controlling-afp}
Let $(\mathcal M, E) = (\mathcal M_1, E_1) \ast_{\mathcal B} (\mathcal M_2, E_2)$ be a semifinite amalgamated free product von Neumann algebra with semifinite faithful normal trace $\Tr$. Let $p \in \Proj(\mathcal M_1)$ and $\mathcal Q \subset p \mathcal M_1 p$ any von Neumann subalgebra. Assume that there exists a net of unitaries $w_k \in \mathcal U(\mathcal Q)$ such that $\lim_k \|E_{\mathcal B}(x^* w_k y)\|_{2, \Tr} = 0$ for all $x, y \in p \mathcal M_1$.

Then any $\mathcal Q$-$p \mathcal M_1 p$-subbimodule $\mathcal H$ of $\rL^2(p \mathcal M p)$ which has finite dimension as a right $p \mathcal M_1 p$-bimodule must be contained in $\rL^2(p \mathcal M_1 p)$. In particular, $\mathcal N_{p \mathcal M p} (\mathcal Q)\dpr \subset p \mathcal M_1 p$.
\end{theo}

\begin{proof}
Using \cite[Proposition V.2.36]{Ta02}, we denote by $E_{\mathcal M_1} : \mathcal M \to \mathcal M_1$ the unique trace preserving faithful normal conditional expectation which satisfies
$$E_{\mathcal M_1}(x_1 \cdots x_{2 m + 1}) = 0$$
whenever $m \geq 1$, $x_1, x_{2 m + 1} \in \mathcal M_1$, $x_{2 j} \in \mathcal M_2 \ominus \mathcal B$ and $x_{2j + 1} \in \mathcal M_1 \ominus \mathcal B$ for all $1 \leq j \leq m - 1$. Observe that we moreover have $\Tr \circ E_{\mathcal M_1} = \Tr$. We denote by $\mathcal M \ominus \mathcal M_1$ the kernel of the conditional expectation $E_{\mathcal M_1} : \mathcal M \to \mathcal M_1$.

\begin{claim}
We have that $\lim_k \|E_{\mathcal M_1}(x^* w_k y)\|_{2, \Tr} = 0$ for all $x, y \in p (\mathcal M \ominus \mathcal M_1)$.
\end{claim}

\begin{proof}[Proof of the Claim]
Observe that using Kaplansky's density theorem, it suffices to prove the Claim for $x = p x_1 \cdots x_{2m + 1}$ and $y = p y_1 \cdots y_{2n + 1}$ with $m, n \geq 1$, $x_1, x_{2m + 1}, y_1, y_{2n + 1} \in \mathcal M_1$, $x_{2 \ell + 1}, y_{2 \ell' + 1} \in \mathcal M_1 \ominus \mathcal B$ and $x_{2 \ell}, y_{2 \ell'} \in \mathcal M_2 \ominus \mathcal B$ for all $1 \leq \ell \leq m - 1$ and all $1\leq  \ell' \leq n - 1$. Then, we have
$$E_{\mathcal M_1}(x^* w_k y) = E_{\mathcal M_1}(x_{2m + 1}^* \cdots x_2^* \, E_{\mathcal B}( x_1^* w_k y_1) \, y_2 \cdots y_{2n + 1}).$$
Hence, $\lim_k \|E_{\mathcal M_1}(x^* w_k y)\|_{2, \Tr} = 0$.
\end{proof}
In particular, we get $\lim_k \|E_{p \mathcal M_1 p}(x^* w_k y)\|_{2, \Tr} = 0$ for all $x, y \in p \mathcal M p \ominus p \mathcal M_1 p$. Finally, applying \cite[Lemma D.3]{vaes-bourbaki-popa}, we are done.
\end{proof}

We will moreover need the following technical results.
 
\begin{prop}\label{intertwining-afp}
Let $(\mathcal M, E) = (\mathcal M_1, E_1) \ast_{\mathcal B} (\mathcal M_2, E_2)$ be a semifinite amalgamated free product von Neumann algebra with semifinite faithful normal trace $\Tr$. Assume the following:
\begin{itemize}
\item For all $i \in \{1, 2\}$ and all nonzero projections $z \in \mathcal Z(\mathcal B)$, $\mathcal Bz \neq z \mathcal M_i z$. 
\item For all $p \in \Proj(\mathcal M)$ and all $q \in \Proj(\mathcal B)$, we have $p \mathcal M p \npreceq_{\mathcal M} q \mathcal B q$.
\end{itemize}

Then for all $i \in \{1, 2\}$, all $e \in \Proj(\mathcal M)$ and all $f \in \Proj(\mathcal M_i)$, we have $e \mathcal M e \npreceq_{\mathcal M} f \mathcal M_i f$.
\end{prop}

\begin{proof}
By contradiction, assume that there exist $i \in \{1, 2\}$, $e \in \Proj(\mathcal M)$ and $f \in \Proj(\mathcal M_i)$, a nonzero partial isometry $v \in e \mathcal M f$ and a unital normal $\ast$-homomorphism $\varphi : e \mathcal M e \to f \mathcal M_i f$ such that $ xv = v \varphi(x)$ for all $x \in e\mathcal M e$. We may assume without loss of generality that $i = 1$. Moreover, as in \cite[Remark 3.8]{vaes-bimodules}, we may assume that the support projection of $E_{\mathcal M_1}(v^*v)$ in $\mathcal M_1$ equals $f$.

Let $q \in \Proj(\mathcal B)$ be arbitrary. By assumption, we have $e \mathcal M e \npreceq_{\mathcal M} q\mathcal B q$. Next, we claim that $\varphi(e \mathcal M e) \npreceq_{\mathcal M_1} q \mathcal B q$. Indeed, otherwise there would exist $n \geq 1$, a nonzero partial isometry $w \in \mathbf M_{1, n} (\C) \otimes f \mathcal M_1 q$ and a normal $\ast$-homomorphism $\psi : \varphi(e \mathcal M e) \to \mathbf M_n(q \mathcal B q)$ such that $\varphi(x) w = w \psi(\varphi(x))$ for all $x \in e \mathcal M e$. Hence, we get $x v w = v w (\psi \circ \varphi)(x)$ for all $x \in e \mathcal M e$. We have $E_{\mathbf M_n(\mathcal M_1)}(w^* v^* v w) = w^* E_{\mathcal M_1}(v^*v) w \neq 0$ since the support projection of $E_{\mathcal M_1}(v^*v)$ is $f$ and $f w = w$. By taking the polar part of $vw$, this would imply that  $e \mathcal M e \preceq_{ \mathcal M} q \mathcal B q$, a contradiction.

By Lemma \ref{intertwining-general-HR} and Theorem \ref{controlling-afp}, we get $\varphi(e \mathcal M e)' \cap f \mathcal M f \subset f \mathcal M_1 f$, hence $v^*v \in f \mathcal M_1 f$. Thus, we may assume that $v^*v = f$. We get $f \mathcal M f = v^* \mathcal M v \subset f \mathcal M_1 f  \subset f \mathcal M f $, so $f \mathcal M_1 f = f \mathcal M f$. The proof of \cite[Theorem 5.7]{houdayer-vaes} shows that there exists a nonzero projection $z \in \mathcal Z(\mathcal B)$ such that $z \mathcal M_2 z = \mathcal B z$, contradicting the assumptions.
\end{proof}

\begin{prop}\label{intertwining-afp-bis}
Let $(\mathcal M, E) = (\mathcal M_1, E_1) \ast_{\mathcal B} (\mathcal M_2, E_2)$ be a semifinite amalgamated free product von Neumann algebra with semifinite faithful normal trace $\Tr$. Let $p \in \Proj(\mathcal M)$ and $\mathcal A \subset p \mathcal M p$ any von Neumann subalgebra. Assume there exist $i \in \{1, 2\}$ and $p_i \in \Proj(\mathcal M_i)$ such that $\mathcal A \preceq_{\mathcal M} p_i \mathcal M_i p_i$.

Then either there exists $q \in \Proj(\mathcal B)$ such that $\mathcal A \preceq_{\mathcal M} q\mathcal B q$ or $\mathcal N_{p \mathcal M p}(\mathcal A)\dpr \preceq_{\mathcal M} p_i \mathcal M_i p_i$.
\end{prop}

\begin{proof}
We assume that for all $q \in \Proj(\mathcal B)$, we have $\mathcal A \npreceq_{\mathcal M} q\mathcal B q$ and show that necessarily $\mathcal N_{p \mathcal M p}(\mathcal A)\dpr \preceq_{\mathcal M} p_i \mathcal M_i p_i$.

Since $\mathcal A \preceq_{\mathcal M} p_i \mathcal M_i p_i$, there exist $n \geq 1$, a nonzero partial isometry $v \in \mathbf M_{1, n}(\C) \otimes p \mathcal M p_i$ and a possibly nonunital normal $\ast$-homomorphism $\varphi :  \mathcal A  \to \mathbf M_n( p_i \mathcal M_i p_i )$ such that $a v= v \varphi(a)$ for all $a \in \mathcal A$. Since we also have $\mathcal A \npreceq_{\mathcal M} q \mathcal B q$ for all $q \in \Proj(\mathcal B)$, a reasoning entirely analogous to the one of the proof of Proposition \ref{intertwining-afp} allows us to further assume that $\varphi(\mathcal A) \npreceq_{\mathbf M_n(\mathcal M_i)} \mathbf M_n(q \mathcal B q)$ for all $q \in \Proj(\mathcal B)$.

Let $u \in \mathcal N_{p \mathcal M p}(\mathcal A)$. Then for all $a \in \mathcal A$, we have
$$v^* u v \varphi(a) = v u a v = v^*(u a u^*) uv = \varphi(u a u^*) v^* u v.$$
By Theorem \ref{controlling-afp} and Lemma \ref{intertwining-general-HR}, we get $v^* u v \in \mathbf M_n(p_i \mathcal M_i p_i)$ for all $u \in \mathcal N_{p \mathcal M p}(\mathcal A)$, hence $v^* \mathcal N_{p \mathcal M p}(\mathcal A)\dpr v \subset p_i \mathcal M_i p_i$. Therefore, we have $\mathcal N_{p \mathcal M p}(\mathcal A)\dpr \preceq_{\mathcal M} p_i \mathcal M_i p_i$.
\end{proof}

\subsection{Hilbert bimodules}

Let $M$ and $N$ be any von Neumann algebras. Recall that an $M$-$N$-{\em bimodule} $\mathcal H$ is a Hilbert space endowed with two commuting normal $\ast$-representations $\pi : M \to \mathbf B(\mathcal H)$ and $\rho : N^{\op} \to \mathbf B(\mathcal H)$. We then define $\pi_{\mathcal H} : M \otimes_{\alg} N^{\op} \to \mathbf B(\mathcal H)$ by $\pi_{\mathcal H}(x \otimes y^{\op}) = \pi(x) \rho(y^{\op})$ for all $x \in M$ and all $y \in N$. We will simply write $x \xi y = \pi_{\mathcal H}(x \otimes y^{\op}) \xi$ for all $x \in M$, all $y \in N$ and all $\xi \in \mathcal H$. 

Let $\mathcal H$ and $\mathcal K$ be $M$-$N$-bimodules. Following \cite[Appendix V.B]{Co94}, we say that $\mathcal K$ is {\em weakly contained} in $\mathcal H$ and write $\mathcal K \subset_{\weak} \mathcal H$ if $\|\pi_{\mathcal K}(T) \|_\infty \leq \|\pi_{\mathcal H}(T)\|_\infty$ for all $T \in M \otimes_{\alg} N^{\op}$. We simply denote by $(N, \rL^2(N), J, \mathfrak P)$ the standard form of $N$ (see e.g.\ \cite[Chapter IX.1]{Ta03}). Then the $N$-$N$-bimodule $\rL^2(N)$ with left and right action given by $x \xi y = x J y^*J \xi$ is the {\em trivial} $N$-$N$-bimodule while the $N$-$N$-bimodule $\rL^2(N) \otimes \rL^2(N)$ with left and right action given by $x (\xi \otimes \eta) y = x \xi \otimes J y^* J \eta$ is the {\em coarse} $N$-$N$-bimodule.

Recall that a von Neumann algebra $N$ is {\em amenable} if as $N$-$N$-bimodules, we have $\rL^2(N) \subset_{\weak} \rL^2(N) \otimes \rL^2(N)$. Equivalently, there exists a norm one projection $\Phi : \mathbf B(\rL^2(N)) \to N$.

For any von Neumann algebras $B, M, N$, any $M$-$B$-bimodule $\mathcal H$ and any $B$-$N$-bimodule $\mathcal K$, there is a well defined $M$-$N$-bimodule $\mathcal H \otimes_B \mathcal K$ called the Connes's fusion tensor product of $\mathcal H$ and $\mathcal K$ over $B$.
We refer to \cite[Appendix V.B]{Co94} and \cite[Section 1]{AD93} for more details regarding this construction.

We will be using the following well known fact (see \cite[Lemma 1.7]{AD93}). For any von Neumann algebras $B, M, N$ such that $B$ is amenable, any $M$-$B$-bimodule $\mathcal H$ and any $B$-$N$-bimodule $\mathcal K$, we have, as $M$-$N$-bimodules,
$$\mathcal H \otimes_B \mathcal K \subset_{\weak} \mathcal H \otimes \mathcal K.$$

\subsection{Relative amenability}

Let $M$ be any von Neumann algebra. Denote by $(M, \rL^2(M), J, \mathfrak P)$ the standard form of $M$.  Let $P \subset 1_PM1_P$ (resp.\ $Q \subset M$) be a von Neumann subalgebra with expectation $E_P : 1_P M 1_P \to P$ (resp.\ $E_Q : M \to Q$). The {\em basic construction} $\langle M, Q\rangle$  is the von Neumann algebra $(J Q J)' \cap \mathbf B(H)$. Following \cite[Section 2.1]{OP07}, we say that $P$ is {\em amenable relative to} $Q$ {\em inside} $M$ if there exists a norm one projection $\Phi : 1_P \langle M, Q\rangle 1_P \to P$ such that $\Phi | 1_P M 1_P = E_P$.

In the case when $(M, \tau)$ is a tracial von Neumann algebra and the conditional expectation $E_P : M \to P$ (resp.\ $E_Q : M \to Q$) is $\tau$-preserving, the basic construction that we denote by $\langle M, e_Q \rangle$ coincides with the von Neumann algebra generated by $M$ and the orthogonal projection $e_Q : \rL^2(M, \tau) \to \rL^2(Q, \tau | Q)$. Observe that $\langle M, e_Q \rangle$ comes with a semifinite faithful normal trace given by $\Tr(x e_Q y) = \tau(xy)$ for all $x, y \in M$. Then \cite[Theorem 2.1]{OP07} shows that $P$ is amenable relative to $Q$ inside $M$ if and only if there exists a net of vectors $\xi_n \in \rL^2(\langle M, e_Q\rangle, \Tr)$ such that $\lim_n \langle x \xi_n, \xi_n\rangle_{\Tr} = \tau(x)$ for all $x \in 1_P M 1_P$ and $\lim_n \|y \xi_n - \xi_n y\|_{2, \Tr} = 0$ for all $y \in P$.

\subsection{Noncommutative flow of weights}

Let $(M, \varphi)$ be a von Neumann algebra together with a faithful normal state. Denote by $M^\varphi$ the centralizer of $\varphi$ and by $M \rtimes_\varphi \R$ the {\it continuous core} of $M$, that is, the crossed product of $M$ with the modular automorphism group $(\sigma_t^\vphi)_{t \in \R}$ associated with the faithful normal state $\varphi$. We have a canonical $\ast$-embedding $\pi_\vphi : M \recht M \rtimes_\vphi \R$ and a canonical group of unitaries $(\lambda_\vphi(s))_{s \in \R}$ in $M \rtimes_\vphi \R$ such that
$$\pi_\vphi(\si^\vphi_s(x)) = \lambda_\vphi(s) \, \pi_\vphi(x) \, \lambda_\vphi(s)^* \quad\text{for all}\;\; x \in M, s \in \R.$$
The unitaries $(\lambda_\vphi(s))_{s \in \R}$ generate a copy of $\rL(\R)$ inside $M \rtimes_\vphi \R$.

We denote by $\vphih$ the \emph{dual weight} on $M \rtimes_\vphi \R$ (see \cite[Definition X.1.16]{Ta03}), which is a semifinite faithful normal weight on $M \rtimes_\vphi \R$ whose modular automorphism group $(\sigma_t^{\vphih})_{t \in \R}$ satisfies
$$
\sigma_t^{\widehat{\varphi}}(\pi_\varphi(x)) = \pi_\varphi(\sigma_t^\varphi(x)) \;\text{for all}\; x \in M \quad\text{and}\quad
\sigma_t^{\widehat{\varphi}}(\lambda_\varphi(s)) = \lambda_\varphi(s) \;\text{for all}\; s \in \R.
$$
We denote by $(\theta^\varphi_t)_{t \in \R}$ the \emph{dual action} on $M \rtimes_\vphi \R$, given by
$$\theta^\vphi_t(\pi_\vphi(x)) = \pi_\vphi(x) \;\text{for all}\; x \in M \quad\text{and}\quad \theta^\vphi_t(\lambda_\vphi(s)) = \exp({\rm i} t s) \lambda_\vphi(s) \;\text{for all}\; s \in \R.$$
Denote by $h_\varphi$ the unique nonsingular positive selfadjoint operator affiliated with $\rL(\R) \subset M \rtimes_\vphi \R$ such that $h_\varphi^{{\rm i}s} = \lambda_\varphi(s)$ for all $s \in \R$. Then $\Tr_\varphi = \widehat{\varphi}(h_\varphi^{-1} \cdot)$ is a semifinite faithful normal trace on $M \rtimes_\varphi \R$ and the dual action $\theta^\varphi$ {\it scales} the trace $\Tr_\varphi$:
\begin{equation*}
\Tr_\varphi \circ \theta^\varphi_t = \exp(t) \Tr_\varphi, \forall t \in \R.
\end{equation*}
Note that $\Tr_\vphi$ is semifinite on $\rL(\R) \subset M \rtimes_\vphi \R$. Moreover, the canonical faithful normal conditional expectation $E_{\rL(\R)} : M \rtimes_\varphi \R \to \rL(\R)$ defined by $E_{\rL(\R)}(x \lambda_\varphi(s)) = \varphi(x) \lambda_\varphi(s)$ preserves the trace $\Tr_\varphi$, that is,
\begin{equation*}
\Tr_\varphi \circ E_{\rL(\R)} = \Tr_\varphi.
\end{equation*}

Because of Connes's Radon-Nikodym cocycle theorem (see \cite[Theorem VIII.3.3]{Ta03}), the semifinite von Neumann algebra $M \rtimes_\vphi \R$, together with its trace $\Tr_\vphi$ and trace-scaling action $\theta^\vphi$, ``does not depend'' on the choice of $\vphi$ in the following precise sense. If $\psi$ is another faithful normal state on $M$, there is a canonical surjective $*$-isomorphism
$\Pi_{\psi,\vphi} : M \rtimes_\vphi \R \recht M \rtimes_\psi \R$ such that $\Pi_{\psi,\vphi} \circ \pi_\vphi = \pi_\psi$, $\Tr_\psi \circ \Pi_{\psi,\vphi} = \Tr_\vphi$ and $\Pi_{\psi,\vphi} \circ \theta^\vphi = \theta^\psi \circ \Pi_{\psi,\vphi}$. Note however that $\Pi_{\psi,\vphi}$ does not map the subalgebra $\rL(\R) \subset M \rtimes_\vphi \R$ onto the subalgebra $\rL(\R) \subset M \rtimes_\psi \R$.

Altogether we can abstractly consider the \emph{continuous core} $(\core(M),\theta,\Tr)$, where $\core(M)$ is a von Neumann algebra with a faithful normal semifinite trace $\Tr$, $\theta$ is a trace-scaling action of $\R$ on $(\core(M),\Tr)$ and $\core(M)$ contains a copy of $M$. Whenever $\vphi$ is a faithful normal state on $M$, we get a canonical surjective $\ast$-isomorphism $\Pi_\vphi : M \rtimes_\vphi \R \recht \core(M)$ such that
$$\Pi_{\varphi} \circ \theta^\varphi = \theta \circ \Pi_{\varphi},\quad \Tr_\varphi = \Tr \circ \Pi_{\varphi}, \quad \Pi_{\varphi}(\pi_\varphi(x)) = x \;\; \forall x \in M.$$
A more functorial construction of the continuous core, known as the {\it noncommutative flow of weights} can be given, see \cite{connes73,connestak,falcone}.

By Takesaki's duality theorem \cite[Theorem X.2.3]{Ta03}, we have that $\core(M) \rtimes_{\theta} \R \cong M \ovt \mathbf B(\rL^2(\R))$. In particular, by \cite[Proposition 3.4]{AD93}, $M$ is amenable if and only if $\core(M)$ is amenable.

If $P \subset 1_P M 1_P$ is a von Neumann subalgebra with expectation, we have a canonical trace preserving inclusion $\core(P) \subset 1_P \core(M) 1_P$.

We will also frequently use the following well-known fact: if $A \subset M$ is a Cartan subalgebra then $\core(A) \subset \core(M)$ is still a Cartan subalgebra. For a proof of this fact, see e.g.\ \cite[Proposition 2.6]{houdayer-ricard}.

\begin{prop}\label{amenable-core}
Let $M$ be any von Neumann algebra with no amenable direct summand. Then the continuous core $\core(M)$ has no amenable direct summand either.
\end{prop}

\begin{proof}
Assume that $\core(M)$ has an amenable direct summand. Let $z \in \mathcal Z(\core(M))$ be a nonzero projection such that $\core(M) z$ is amenable. Denote by $\theta : \R \curvearrowright \core(M)$ the dual action which scales the trace $\Tr$. Put $e = \bigvee_{t \in \R} \theta_t(z)$. Observe that $e \in \mathcal Z(\core (M))$ and $\theta_t(e) = e$ for all $t \in \R$. By \cite[Theorem ${\rm XII}$.6.10]{Ta03}, we have $e \in M \cap \mathcal Z(\core(M))$, hence $e \in \mathcal Z(M)$. We canonically have $\core(M) e = \core(M e)$.

Since amenability is stable under direct limits, we have that $\core(M) e$ is amenable, hence $\core(M e)$ is amenable. Applying again \cite[Theorem ${\rm XII}$.6.10]{Ta03}, we have $\core(M e) \rtimes_\theta \R \cong (M e) \ovt \mathbf B(\rL^2(\R))$. We get that $\core(M e) \rtimes_\theta \R$ is amenable and so is $Me$. Therefore, $M$ has an amenable direct summand.
\end{proof}

We will frequently use the following:

\begin{nota}\label{notation}
Let $A \subset M$ (resp.\ $B \subset M$) be a von Neumann subalgebra with expectation $E_A : M \to A$ (resp.\ $E_B : M \to B$) of a given von Neumann algebra $M$.  Assume moreover that $A$ and $B$ are both tracial. Let $\tau_A$ be a faithful normal trace on $A$
(resp.\ $\tau_B$ on $B$) and write $\varphi_A = \tau_A \circ E_A$ (resp.\ $\varphi_B = \tau_B \circ E_B$). 
Write $\pi_{\varphi_A} : M \to M \rtimes_{\varphi_A} \R$ (resp.\ $\pi_{\varphi_B} : M \to M \rtimes_{\varphi_B} \R$) for the canonical $\ast$-representation of $M$ into its continuous core associated with $\varphi_A$ (resp.\ $\varphi_B$).

By Connes's Radon-Nikodym cocycle theorem, there is a surjective $\ast$-isomorphism
$$\Pi_{\varphi_B, \varphi_A} : M \rtimes_{\varphi_A} \R \to M \rtimes_{\varphi_B} \R$$
which intertwines the dual actions, that is, $\theta^{\varphi_B} \circ \Pi_{\varphi_B, \varphi_A} = \Pi_{\varphi_B, \varphi_A} \circ \theta^{\varphi_A}$, and preserves the faithful normal semifinite traces, that is, $\Tr_{\varphi_B} \circ \Pi_{\varphi_B, \varphi_A} = \Tr_{\varphi_A}$. In particular, we have $\Pi_{\varphi_B, \varphi_A}(\pi_{\varphi_A}(x)) = \pi_{\varphi_B}(x)$ for all $x \in M$. 

Put $\core(M) = M \rtimes_{\varphi_B} \R$, $\core(B) = B \rtimes_{\varphi_B} \R$ and $\core(A) = \Pi_{\varphi_B, \varphi_A}(A \rtimes_{\varphi_A} \R)$. We simply denote by $\Tr = \Tr_{\varphi_B}$ the canonical semifinite faithful normal trace on $\core(M)$.  Observe that $\Tr$ is still semifinite on $\mathcal Z(\core(A))$ and $\mathcal Z(\core(B))$.
\end{nota}

\begin{prop}\label{intertwining-core}
Assume that we are in the setup of Notation \ref{notation}. If $A \npreceq_M B$, then for all $p \in \Proj(\mathcal Z(\core(A)))$ and all $q \in \Proj(\mathcal Z(\core(B)))$, we have $\core(A) p \npreceq_{\core(M)} \core(B) q$.
\end{prop}

\begin{proof}
Let $v_k \in \mathcal U(A)$ be a net such that $E_B(x^* v_k y) \to 0$ $\ast$-strongly for all $x, y \in M$. Recall that $\core(M) = M \rtimes_{\varphi_B} \R$, $\core(B) = B \rtimes_{\varphi_B} \R$ and $\core(A) = \Pi_{\varphi_B, \varphi_A}(A \rtimes_{\varphi_A} \R)$. Let $p \in \Proj(\mathcal Z(\core(A)))$ and $q \in \Proj(\mathcal Z(\core(B)))$. Observe that since $p$ commutes with every element in $\core(A)$, $p$ commutes with every element in $\Pi_{\varphi_B, \varphi_A}(\pi_{\varphi_A}(A)) = \pi_{\varphi_B}(A) \subset \core(A)$. Then $w_k = \Pi_{\varphi_B, \varphi_A}(\pi_{\varphi_A}(v_k)) p = \pi_{\varphi_B}(v_k)p$ is a net of unitaries in $\mathcal U(\core(A) p)$.

Write $\core(M)_{\alg} = M \rtimes_{\varphi_B}^{\alg} \R$ for the algebraic crossed product, that is, the linear span of $\{\pi_{\varphi_B}(x) \lambda_{\varphi_B}(t) : x \in M, t \in \R\}$. Observe that $\core(M)_{\alg}$ is a dense unital $\ast$-subalgebra of $\core(M)$. We have $E_{\core(B)}(x^* \pi_{\varphi_B}(v_k) y) \to 0$ $\ast$-strongly for all $x, y \in \core(M)_{\alg}$. Since $q \in \Proj(\core(B))$, we have
$$\|E_{ \core(B) q}(q \, x^* \pi_{\varphi_B}(v_k)y \, q) \|_{2, \Tr} = \|q E_{\core(B)}(x^* \pi_{\varphi_B}(v_k) y) q \|_{2, \Tr} \to 0, \forall x, y \in \core(M)_{\alg}.$$

Fix now $x, y \in \Ball(\core(M))$. By Kaplansky density theorem, choose a net $(x_i)_{i \in I}$ (resp.\ $(y_j)_{j \in J}$) in $\Ball(\core(M)_{\alg})$ such that $x_i \to px$ (resp.\ $y_j \to py$) $\ast$-strongly. Let $\varepsilon > 0$. Since $q \in \Proj(\core(B))$, we can choose $(i, j) \in I \times J$ such that 
$$\|(py - y_j) q\|_{2, \Tr} + \|q (x^*p - x_j)\|_{2, \Tr} < \varepsilon.$$
Therefore, by triangle inequality, we obtain
$$
\limsup_k \|E_{\core(B) q}(q \, x^* \, p\pi_{\varphi_B}(v_k)p \, y \, q) \|_{2, \Tr}  \leq \limsup_k \|E_{ \core(B) q}(q \, x_i^*\pi_{\varphi_B}(v_k) y_j \, q) \|_{2, \Tr} +  \varepsilon \leq \varepsilon.$$
Since $\varepsilon > 0$ is arbitrary, we get $\lim_k \|E_{\core(B) q}(q x^* p \, w_k \, p y  q) \|_{2, \Tr} = 0$. This finally proves that $\core(A) p \npreceq_{\core(M)} \core(B) q$.
\end{proof}

\begin{example}
We emphasize two well-known examples that will be extensively used in this paper.
\begin{enumerate}
\item Let $\Gamma \curvearrowright (X, \mu)$ be any nonsingular action on a standard measure space. Define the \emph{Maharam extension} (see \cite{maharam}) $\Gamma \curvearrowright (X \times \R, m)$ by
$$g \cdot (x, t) = \left( gx, t + \log\left( \frac{{\rm d} (\mu \circ g^{-1})}{{\rm d} \mu}(x)\right)\right),$$
where ${\rm d}m = {\rm d} \mu \times \exp(t){\rm d}t$. It is easy to see that the action $\Gamma \curvearrowright X \times \R$ preserves the infinite measure $m$ and we moreover have that
$$\core(\rL^\infty(X) \rtimes \Gamma) = \rL^\infty(X \times \R) \rtimes \Gamma.$$

\item Let $(M, E) = (M_1, E_1) \ast_B (M_2, E_2)$ be any amalgamated free product von Neumann algebra.  Fix a faithful normal state $\varphi$ on $B$. We still denote by $\varphi$ the faithful normal state $\varphi \circ E$ on $M$. We realize the continuous core of $M$ as $\core(M) = M \rtimes_\vphi \R$. Likewise, if we denote by $\varphi_i = \varphi \circ E_i$ the corresponding state on $M_i$, we realize the continuous core of $M_i$ as $\core(M_i) = M_i \rtimes_{\vphi_i} \R$. We denote by $\core(E) : \core(M) \to \core(B)$ (resp.\ $\core(E_i) : \core(M_i) \to \core(B)$) the canonical trace preserving faithful normal conditional expectation. Recall from \cite[Section 2]{ueda-pacific} that $\si_t^\vphi(M_i) = M_i$ for all $t \in \R$ and all $i \in \{1, 2\}$, hence
$$(\core(M), \core(E)) = (\core(M_1), \core(E_1)) \ast_{\core(B)} (\core(M_2), \core(E_2)).$$
Moreover, $\core(M)$ is a semifinite amalgamated free product von Neumann algebra.
\end{enumerate}
\end{example}

\section{Intertwining subalgebras inside semifinite AFP von Neumann algebras}

\subsection{Malleable deformation on semifinite AFP von Neumann algebras}\label{section-malleable}
First, we recall the construction of the malleable deformation on amalgamated free product von Neumann algebras discovered in \cite[Section 2]{IPP05}.

Let $(\mathcal M, E) = (\mathcal M_1, E_1) \ast_{\mathcal B} (\mathcal M_2, E_2)$ be any semifinite amalgamated free product von Neumann algebra with semifinite faithful normal trace $\Tr$. We will simply write $\mathcal M = \mathcal M_1 \ast_{\mathcal B} \mathcal M_2$ when no confusion is possible. Put $\widetilde {\mathcal M} = \mathcal M \ast_{\mathcal B} (\mathcal B \ovt \rL(\F_2))$ and observe that $\widetilde{\mathcal M}$ is still a semifinite amalgamated free product von Neumann algebra. We still denote by $\Tr$ the semifinite faithful normal trace on $\widetilde{\mathcal M}$. Let $u_1, u_2 \in \mathcal U(\rL(\mathbf F_2))$ be the canonical Haar unitaries generating $\rL(\F_2)$. Observe that we can decompose $\widetilde {\mathcal M} = \widetilde {\mathcal M}_1 \ast_{\mathcal B} \widetilde {\mathcal M}_2$ with $\widetilde {\mathcal M}_i = \mathcal M_i \ast_{\mathcal B} (\mathcal B \ovt \rL(\Z))$.

Consider the unique Borel function $f : \mathbf T \to (-\pi, \pi]$ such that $f(\exp({\rm i} t)) = t$
 for all $t \in (-\pi, \pi]$. Define the selfadjoint operators $h_1 = f(u_1)$ and $h_2 = f(u_2)$ so that $\exp({\rm i} u_1) = h_1$ and $\exp({\rm i} u_2) = h_2$. For every $t \in \R$, put $u_1^t = \exp({\rm i} t h_1)$ and $u_2^t = \exp({\rm i} t h_2)$. We have 
$$\tau(u_1^t) = \tau(u_2^t) =  \frac{\sin(\pi t)}{\pi t}, \forall t \in \R.$$
 
Define the one-parameter group of trace preserving $\ast$-automorphisms $\alpha_t \in \Aut(\widetilde {\mathcal M})$ by
$$\alpha_t = \Ad(u_1^t) \ast_{\mathcal B} \Ad(u_2^t), \forall t \in \R.$$
Define moreover the trace preserving $\ast$-automorphism $\beta \in \Aut(\widetilde{\mathcal M})$ by 
$$\beta = \id_{\mathcal M} \ast_\mathcal B (\id_{\mathcal B} \ovt \beta_0)$$
with $\beta_0(u_1) = u_1^*$ and $\beta_0(u_2) = u_2^*$. We have $\alpha_t \beta = \beta \alpha_{- t}$ for all $t \in \R$. Thus, $(\alpha_t, \beta)$ is a malleable deformation in the sense of Popa \cite{Po06b}.

We will be using the following notation throughout this section. 

\begin{nota}\label{notation2}
Put $\mathcal H_0 = \rL^2(\mathcal B, \Tr)$ and $\mathcal K_0 = \rL^2(\cB \ovt \rL(\F_2), \Tr)$. For $n \geq 1$, define $S_n = \lbrace (i_1 , \dots, i_n) : i_1 \neq \cdots \neq i_n \rbrace$ to be the set of the two alternating sequences of length $n$ made of $1$'s and $2$'s. For $\mathcal I = (i_1, \ldots, i_n) \in S_n$, denote by 
\begin{itemize}
\item $\mathcal H_{\mathcal I}$ the closed linear span in $\rL^2(\mathcal M, \Tr)$ of elements $x_1 \cdots x_n$, with  $x_j \in \cM_{i_j} \ominus \cB$ such that $\Tr(x_j^*x_j) < \infty$ for all $j \in \lbrace 1, \dots, n\rbrace$.
\item $\mathcal K_{\mathcal I}$ the closed linear span in $\rL^2(\widetilde{\mathcal M}, \Tr)$ of elements $u_{h_1}x_1 \cdots u_{h_n} x_n u_{h_{n+1}}$, with $h_j \in \F_2$ for all $j \in \{1, \dots, n + 1\}$ and $x_j \in \cM_{i_j} \ominus \cB$ such that $\Tr(x_j^*x_j) < \infty$ for all $j \in \lbrace 1, \dots, n\rbrace$.
\end{itemize}
We denote by $E_{\mathcal M} : \widetilde{\mathcal M} \to \mathcal M$ the unique trace preserving faithful normal conditional expectation as well as the orthogonal projection $\rL^2(\widetilde{\mathcal M}, \Tr) \to \rL^2(\mathcal M, \Tr)$. We still denote by $\alpha : \R \to \mathcal U(\rL^2(\widetilde {\mathcal M}, \Tr))$ the Koopman representation associated with the trace preserving action $\alpha : \R \to \Aut(\widetilde{\mathcal M})$.
\end{nota}

\begin{lem}
\label{orthogonalranges}
Let $m, n \geq 1$, $\cI = (i_1,\ldots,i_m) \in S_m$ and $\cJ = (j_1,\ldots,j_n) \in S_n$. Let $x_1 \in \cM_{i_1} \ominus \cB, \ldots, x_m \in \cM_{i_m} \ominus \cB$ and $y_1 \in \cM_{j_1} \ominus \cB, \ldots, y_n \in \cM_{j_n} \ominus \cB$ with $\Tr(a^*a) < \infty$ for $a = x_1,\ldots , x_m,y_1,\ldots,y_n$. Let $g_1,\ldots,g_{m+1},h_1,\ldots, h_{n+1} \in \F_2$. Then 
$$\langle u_{g_1}x_1 \cdots u_{g_{m}}x_m u_{g_{m+1}},u_{h_1}y_1\cdots u_{h_{n}}y_n u_{h_{n+1}} \rangle_{\rL^2(\widetilde{\mathcal M},\Tr)} =$$
$$\begin{cases} \langle x_1 \cdots x_m, y_1 \cdots y_n \rangle_{\rL^2(\mathcal M, \Tr)} \mbox{ if } m = n, \, \mathcal I=\mathcal J \mbox{ and } g_k = h_k,\forall k\in\{1,\ldots ,m+1\}; \\
0 \mbox{ otherwise}. \end{cases}$$
\end{lem}
\begin{proof}
The proof is the same as the proof of \cite[Lemma 3.1]{Io12a}. We leave it to the reader.
\end{proof}

Lemma \ref{orthogonalranges} allows us, in particular, to put $\mathcal H_n = \bigoplus_{\mathcal I \in S_n} \mathcal H_{\mathcal I}$ and $\mathcal K_n = \bigoplus_{\mathcal I \in S_n} \mathcal K_{\mathcal I}$ since the $\mathcal K_{\mathcal I}$'s are pairwise orthogonal. We then have
$$\rL^2(\mathcal M, \Tr) = \bigoplus_{n \in \N} \mathcal H_n \; \mbox{  and  } \; \rL^2(\widetilde{\mathcal M}, \Tr) = \bigoplus_{n \in \N} \mathcal K_n.$$

For all $\xi \in \rL^2(\mathcal M, \Tr)$, write $\xi = \sum_{n \in \N} \xi_n$ with $\xi_n \in \mathcal H_n$ for all $n \in \N$. A simple calculation shows that for all $t \in \R$, 
$$\Tr(\alpha_t(\xi) \xi^*) = \Tr(E_{\mathcal M}(\alpha_t(\xi)) \xi^*) = \sum_{n \in \N} \left( \frac{\sin(\pi t)}{\pi t} \right)^{2n} \|\xi_n\|_{2, \Tr}^2.$$
Observe that $t \mapsto \Tr(\alpha_t(\xi) \xi^*)$ is decreasing on $[0, 1]$ for all $\xi \in \rL^2(\mathcal M, \Tr)$.

\subsection{A semifinite analogue of Ioana-Peterson-Popa's intertwining theorem \cite{IPP05}}
The first result of this section is an analogue of the main technical result of \cite{IPP05} (see \cite[Theorem 4.3]{IPP05}) for semifinite amalgamated free product von Neumann algebras. A similar result also appeared in \cite[Theorem 4.2]{CH08}. For the sake of completeness, we will give the proof.

\begin{theo}\label{convergence-deformation}
Let $\mathcal M = \mathcal M_1 \ast_{\mathcal B} \mathcal M_2$ be a semifinite amalgamated free product von Neumann algebra with semifinite faithful normal trace $\Tr$. Let $p \in \Proj(\mathcal M)$ and $\mathcal A \subset p \mathcal M p$ any von Neumann subalgebra. Assume that there exist $c > 0$ and $t \in (0, 1)$ such that $\Tr(\alpha_t(w) w^*) \geq c$ for all $w \in \mathcal U(\mathcal A)$.

Then there exists $q \in \Proj(\mathcal B)$ such that $\mathcal A \preceq_{\mathcal M} q \mathcal B q$ or there exists $i \in \{1, 2\}$ and $q_i \in \Proj(\mathcal M_i)$ such that $\mathcal N_{p \mathcal M p}(\mathcal A)\dpr \preceq_{\mathcal M} q_i \mathcal M_i q_i$.
\end{theo}
 
\begin{proof}
By assumption, there exist $c > 0$ and $t \in (0, 1)$ such that $\Tr(\alpha_t(w) w^*) \geq c$ for all $w \in \mathcal U(\mathcal A)$. Choose $r \in \N$ large enough such that $2^{-r} \leq t$. Then $\Tr(\alpha_{2^{-r}}(w) w^*) \geq c$ for all $w \in \mathcal U(\mathcal A)$. So, we may assume that $t = 2^{-r}$. A standard functional analysis trick yields a nonzero partial isometry $v \in \alpha_t(p) \widetilde {\mathcal M} p$ such that $v x= \alpha_t(x) v$ for all $x \in \mathcal A$. Observe that $v^*v \in \mathcal A' \cap p  \widetilde{\mathcal M} p$ and $vv^* \in \alpha_t(\mathcal A' \cap  p \widetilde {\mathcal M} p )$.

We prove the result by contradiction. Using Proposition \ref{intertwining-afp-bis} and as in the proof of Proposition \ref{intertwining-semifinite}, we may choose a net of unitaries $w_k \in \mathcal U(\mathcal A)$ such that $\lim_k \|E_{\mathcal M_i}(x^* w_k y)\|_{2, \Tr} = 0$ for all $i \in \{1, 2\}$ and all $x, y \in p \mathcal M$. In particular, we get $\lim_k \|E_{\mathcal B}(x^* w_k y)\|_{2, \Tr} = 0$ for all $x, y \in p \mathcal M$. Regarding $\widetilde{\mathcal M} = \mathcal M \ast_{\mathcal B} (\mathcal B \ovt \rL(\F_2))$, we get $v^*v \in \mathcal A' \cap p \mathcal M p$ by Theorem \ref{controlling-afp}. We use now Popa's malleability trick \cite{popa-malleable1} and put $w = \alpha_t(v \beta(v^*)) \in \alpha_{2t}(p) \widetilde{\mathcal M} p$. Since $ww^* = \alpha_t(vv^*) \neq 0$, we get $w \neq 0$ and $w x =  \alpha_{2t}(x) w$ for all $x \in \mathcal A$. Iterating this construction, we find a nonzero partial isometry $v \in \alpha_1(p) \widetilde {\mathcal M} p$ such that 
\begin{equation}\label{malleability}
v x  =  \alpha_1(x) v, \; \forall x \in \mathcal A.
\end{equation}
Moreover, using again Proposition \ref{controlling-afp}, we get $v^*v \in \mathcal A' \cap p  \mathcal M p$ and $vv^* \in \alpha_1(\mathcal A' \cap p \mathcal M p)$.

Next, exactly as in the proof of \cite[Claim 4.3]{CH08}, we obtain the following.
\begin{claim}
We have $\lim_k \|E_{\alpha_1(\mathcal M)}(x^* w_k y)\|_{2, \Tr} = 0$ for all $x, y \in p \widetilde {\mathcal M}$.
\end{claim}

\begin{proof}[Proof of the Claim]
Regard $\widetilde {\mathcal M} = \mathcal M \ast_{\mathcal B} (\mathcal B \ovt \rL(\F_2))$. By Kaplansky density theorem, it suffices to prove the Claim for $x =  p a$ and $y = p b$ with $a, b$ in $\mathcal B$ or reduced words in $\widetilde {\mathcal M}$ with letters alternating from $\mathcal M \ominus \mathcal B$  and $\mathcal B \ovt \rL(\F_2) \ominus \mathcal B \ovt \C 1$. Write $a = c a'$ with $c = a$ if $a \in \mathcal B$; $c = 1$ if $a$ begins with a letter from $\mathcal B \ovt \rL(\F_2) \ominus \mathcal B \ovt \C 1$; $c$ equals the first letter of $a$ otherwise. Likewise, write $b = d b'$. Then we have $x^* w_k y = a^*  w_k b = a'^* \, c^* w_k d \, b'$ and note that $c^* w_k d \in \mathcal M$. Observe that $a'$ (resp.\ $b'$) equals $1$ or is a reduced word beginning with a letter from $\mathcal B \ovt \rL(\F_2) \ominus \mathcal B \ovt \C 1$.

Denote by $P$ the orthogonal projection from $\rL^2(\mathcal M, \Tr)$ onto $\mathcal H_0 \oplus \mathcal H_1$. Observe that since $c^* w_k d \in \mathcal M \cap \rL^2(\mathcal M, \Tr)$, we have
$$P(c^* w_k d) = E_{\mathcal M_1}(c^* w_k d) + E_{\mathcal M_2}(c^* w_k d) - E_{\mathcal B}(c^* w_k d).$$
Hence, $\lim_k \|P(c^* w_k d)\|_{2, \Tr} = 0$. Moreover, a simple calculation shows that 
$$E_{\alpha_1(\mathcal M)}(x^* w_k y) = E_{\alpha_1(\mathcal M)}(a'^* P(c^* w_k d) b').$$
Therefore, $\lim_k \|E_{\alpha_1(\mathcal M)}(x^* w_k y)\|_{2, \Tr} = 0$. This finishes the proof of the Claim.
\end{proof}

Finally, combining Equation $(\ref{malleability})$ together with the Claim, we get
$$\|vv^*\|_{2, \Tr} = \|\alpha_1(w_k) vv^*\|_{2, \Tr} = \|E_{\alpha_1(\mathcal M)}(\alpha_1(w_k) vv^*)\|_{2, \Tr} = \|E_{\alpha_1(\mathcal M)}(v w_k v^*)\|_{2, \Tr} \to 0.$$
This contradicts the fact that $v \neq 0$ and finishes the proof of Theorem \ref{convergence-deformation}.
\end{proof}

\subsection{A semifinite analogue of Ioana's intertwining theorem \cite{Io12a}}
Let $\mathcal M = \mathcal M_1 \ast_{\mathcal B} \mathcal M_2$ be a semifinite amalgamated free product von Neumann algebra with semifinite faithful normal trace $\Tr$. Put $\widetilde {\mathcal M} = \mathcal M \ast_{\mathcal B} (\mathcal B \ovt \rL(\F_2))$ and observe that $\widetilde{\mathcal M}$ is still a semifinite amalgamated free product von Neumann algebra. We still denote by $\Tr$ the semifinite faithful normal trace on $\widetilde{\mathcal M}$. Let $\mathcal N = \bigvee \{ u_g \mathcal M u_g^* : g \in \F_2\} \subset \widetilde{\mathcal M}$. Observe that $\mathcal N$ can be identified with an infinite amalgamated free product von Neumann algebra, that $\Tr | \mathcal N$ is semifinite and that, under this identification, the action $\F_2 \curvearrowright \mathcal N$ is given by the {\em free Bernoulli shift} which preserves the canonical trace $\Tr$. We moreover have $\widetilde{\mathcal M} = \mathcal N \rtimes \F_2$. 

We will denote by $E_{\mathcal N} : \widetilde{\mathcal M} \to \mathcal N$ the unique trace preserving faithful normal conditional expectation as well as the orthogonal projection  $E_{\mathcal N} : \rL^2(\widetilde{\mathcal M}, \Tr) \to \rL^2(\mathcal N, \Tr)$.

We prove next the analogue of \cite[Theorem 3.2]{Io12a} for semifinite amalgamated free product von Neumann algebras.

\begin{theo}\label{dichotomy-intertwining}
Let $\mathcal M = \mathcal M_1 \ast_{\mathcal B} \mathcal M_2$ be a semifinite amalgamated free product von Neumann algebra with semifinite faithful normal trace $\Tr$. Let $p \in \Proj(\mathcal B)$, $\mathcal A \subset p\mathcal M p$ any von Neumann subalgebra and $t \in (0, 1)$. Assume that there is no net of unitaries $w_k \in \mathcal U(\mathcal A)$ such that 
\[ \lim_k \Vert E_{\mathcal N}(x^*\alpha_t(w_k)y)\Vert_{2,\Tr} = 0, \, \forall x, y \in p\widetilde{\mathcal M}.\]

Then there exists $q \in \Proj(\mathcal B)$ such that $\mathcal A \preceq_{\mathcal M} q \mathcal B q$ or there exists $i \in \{1, 2\}$ and $q_i \in \Proj(\mathcal M_i)$ such that $\mathcal N_{p \mathcal M p}(\mathcal A)\dpr \preceq_{\mathcal M} q_i \mathcal M_i q_i$.
\end{theo}

The main technical lemma that will be used to prove Theorem \ref{dichotomy-intertwining} is a straightforward generalization of \cite[Lemma 3.4]{Io12a}. We include a proof for the sake of completeness.

\begin{lem}\label{cn}
Let $t \in (0,1)$ and $g,h \in \F_2$. For all $n \geq 0$, define 
$$c_n = \sup_{x \in \mathcal H_n, \, \| x \|_{2,\Tr} \leq 1} \| E_\cN(u_g\alpha_t(x)u_h) \|_{2,\Tr}.$$
Then  $\lim_n c_n = 0$.
\end{lem}

\begin{proof}
First, observe that for all $g_1 \ldots,g_{n+1} \in \F_2$ and all $x_1,\ldots, x_n \in \cM$, we have
\begin{equation}\label{esp} 
E_{\cN}(u_{g_1}x_1\cdots u_{g_n}x_nu_{g_{n+1}}) = \begin{cases} u_{g_1}x_1\cdots u_{g_n}x_nu_{g_{n+1}} \mbox{ if } g_1 \cdots g_{n+1} = 1; \\ 
 0 \mbox{ otherwise}.\end{cases}
 \end{equation}
Thus for all $\mathcal I \in S_n$, we have $E_\cN(\cK_\cI) \subset \cK_\cI$ and since $\alpha_t(\mathcal H_{\mathcal I}) \subset \mathcal K_{\mathcal I}$, we get that $E_\cN(u_g\alpha_t(x)u_h) \in \mathcal K_{\mathcal I}$ for all $x \in \mathcal H_{\mathcal I}$. So defining 
$$c_{\mathcal I}  = \sup_{x \in \mathcal H_{\mathcal I}, \, \Vert x \Vert_{2,\Tr} \leq 1} \Vert E_\cN(u_g\alpha_t(x)u_h) \Vert_{2,\Tr},$$ 
we see that $c_n = \max_{\mathcal I \in S_n} c_{\mathcal I}$ since the subspaces $\mathcal K_{\mathcal I}$'s are pairwise orthogonal.

Let us fix $\mathcal I = (i_1,\ldots,i_n) \in S_n$ and calculate $c_{\mathcal I}$. Denote by $a$ and $b$ the canonical generators of $\F_2$ so that $u_1 = u_a$, $u_2 = u_b$ and put $G_1 = \langle a \rangle$ and $G_2= \langle b \rangle$. For $g_1,h_1 \in G_{i_1},\ldots, g_n,h_n \in G_{i_n}$, define a map
$$V_{g_1,h_1, \dots, g_n,h_n}(x_1\cdots x_n) = u_{g_1}x_1u_{h_1}^* \cdots u_{g_n} x_n u_{h_n}^*$$
for all $x_j \in \cM_{i_j} \ominus \cB$ such that $\Tr(x_j^*x_j) < \infty$ for all $j \in \lbrace 1, \dots, n\rbrace$. By Lemma \ref{orthogonalranges}, these maps $V_{g_1,h_1,\ldots,g_n,h_n}$ extend to isometries $V_{g_1,h_1,\ldots,g_n,h_n} : \mathcal H_{\mathcal I} \rightarrow \mathcal K_{\mathcal I}$ with pairwise orthogonal ranges when $(g_1, h_1, \dots, g_n, h_n)$ are pairwise distinct. Indeed, we have $V_{g_1,h_1, \dots, g_n,h_n}(\mathcal H_{\mathcal I}) \perp V_{g_1',h_1', \dots, g_n',h_n'}(\mathcal H_{\mathcal I})$ unless $g_1 = g_1', h_1^{-1} g_2 = h_1'^{-1} g_2', \dots, h_{n - 1}^{-1}g_n = h_{n - 1}'^{-1} g_n', h_n^{-1} = h_n'^{-1}$. Since moreover $G_1 \cap G_2 = \{e\}$, this further implies that $g_j = g_j'$ and $h_j = h_j'$ for all $j \in \{1, \dots, n\}$.

Denote the Fourier coefficients of $u_1^t$ and $u_2^t$ respectively by $\beta_1(g_1)=\tau(u_1^tu_{g_1}^*)$ for  $g_1 \in G_1$ and $\beta_2(g_2)=\tau(u_2^tu_{g_2}^*)$ for $g_2 \in G_2$. We have an explicit formula for these coefficients given by
$$\beta_i(u_i^n) = \tau(u_i^t u_i^{-n}) = \tau(u_i^{t - n}) = \frac{\sin(\pi (t - n))}{\pi(t - n)}.$$
It follows in particular that $\beta_i(g_i) \in \R$ for all $i \in \{1, 2\}$ and all $g_i \in G_i$.
 Since $u_1^t$ and $u_2^t$ are unitaries, we moreover have  \[\sum_{g_1\in G_1}\beta_1(g_1)^2=\sum_{g_2\in G_2}\beta_2(g_2)^2=1.\]

If $x =x_1\cdots x_n$ with $x_j\in \cM_{i_j}\ominus \cB$ satisfying $\Tr(x_j^*x_j) < \infty$, we have
\begin{align*}
u_g\alpha_t(x)u_h & = u_g \, u_{i_1}^tx_1u_{i_1}^{t*}\cdots u_{i_n}^tx_nu_{i_n}^{t*} \,u_h\\
& = \sum_{g_1,h_1\in G_{i_1},\ldots,g_n, h_n\in G_{i_n}}\beta_{i_1}(g_1)\beta_{i_1}(h_1)\cdots \beta_{i_n}(g_n)\beta_{i_n}(h_n)\;u_g \, u_{g_1}x_1u_{h_1}^*\cdots u_{g_n}x_n u_{h_n}^* \, u_h\\
& = \sum_{g_1,h_1\in G_{i_1},\ldots,g_n, h_n\in G_{i_n}}\beta_{i_1}(g_1)\beta_{i_1}(h_1)\cdots \beta_{i_n}(g_n)\beta_{i_n}(h_n)\;u_gV_{g_1,h_1,\ldots,g_n,h_n}(x)u_h,
\end{align*}
where the sum converges in $\Vert \cdot \Vert_{2,\Tr}$. Thus, for all $x \in \mathcal H_{\mathcal I}$, we get
$$u_g\alpha_t(x)u_h = \sum_{g_1,h_1\in G_{i_1},\ldots,g_n, h_n\in G_{i_n}}\beta_{i_1}(g_1)\beta_{i_1}(h_1)\cdots \beta_{i_n}(g_n)\beta_{i_n}(h_n)\;u_gV_{g_1,h_1,\ldots,g_n,h_n}(x)u_h.$$

Now, using the calculation $(\ref{esp})$, and the fact that the isometries $V_{g_1,h_1, \dots ,g_n,h_n}$ have mutually orthogonal ranges, we get that for all $x \in \mathcal H_{\mathcal I}$,
$$\Vert E_\cN(u_g\alpha_t(x)u_h) \Vert_{2,\Tr}^2 = \Vert x \Vert_{2,\Tr}^2 \sum_{\substack{g_1,h_1\in G_{i_1},\ldots ,g_n, h_n\in G_{i_n}\\ gg_1h_1^{-1} \cdots g_n h_n^{-1}h = 1}} \beta_{i_1}(g_1)^2\beta_{i_1}(h_1)^2 \cdots \beta_{i_n}(g_n)^2 \beta_{i_n}(h_n)^2.$$
Thus we get an explicit formula for $c_{\mathcal I}$ given by
\begin{equation}\label{c_I}
c_{\mathcal I} =\sum_{\substack{g_1,h_1\in G_{i_1},\ldots ,g_n,h_n\in G_{i_n}\\ gg_1h_1\cdots g_nh_nh = 1}} \beta_{i_1}(g_1)^2\beta_{i_1}(h_1^{-1})^2 \ldots \beta_{i_n}(g_n)^2\beta_{i_n}(h_n^{-1})^2.
\end{equation}

For $i \in \{1,2\}$, define $\mu_i \in \Prob(\F_2)$ by $\mu_i(g) = \beta_i(g)^2$ if $g \in G_i$ and $\mu_i(g) = 0$ otherwise. Likewise, define $\check \mu_i \in \Prob(\F_2)$ by $\check \mu_i(g) = \mu_i(g^{-1})$ for all $g \in \F_2$. Put $\nu_i = \mu_i \ast \check \mu_i$.
Then we have 
$$c_{\mathcal I} = (\nu_{i_1}  \ast \cdots \ast \nu_{i_n})(g^{-1} h^{-1}).$$
So if we put $\mu = \nu_1 \ast  \nu_2$, we have that 
$$c_{\mathcal I} \in \left\lbrace \mu^{\ast[\frac{n}{2}]}(g^{-1} h^{-1}),\mu^{\ast[\frac{n}{2}]} \ast \nu_1(g^{-1}h^{-1}),\nu_2 \ast \mu^{\ast[\frac{n}{2}]}(g^{-1} h^{-1}),\nu_2  \ast \mu^{\ast[\frac{n-1}{2}]}\ast\nu_1(g^{-1} h^{-1}) \right\rbrace.$$
Then \cite[Lemma 2.13]{Io12a} implies that
$\lim_k \mu^{\ast k}(s) = 0$ for all $s \in \F_2$ and so $\lim_n c_n = 0$.
\end{proof}

\begin{proof}[Proof of Theorem \ref{dichotomy-intertwining}]
Assume by contradiction that the conclusion of the theorem does not hold. Then Theorem \ref{convergence-deformation} implies that for $t \in (0, 1)$ there exists a net $w_k \in \mathcal U(\cA)$ such that 
$$\lim_k \Tr(\alpha_t(w_k)w_k^*) = 0.$$
We will show that for all $x,y \in p\widetilde{\mathcal M}$, we have $\lim_k \Vert E_{\mathcal N}(x^*\alpha_t(w_k)y)\Vert_{2,\Tr} = 0$, which will contradict the assumption of Theorem \ref{dichotomy-intertwining}. By a linearity/density argument, it is sufficient to show that for all $g, h \in \F_2$, 
\begin{equation} \label{absurd}
\lim_k \Vert E_{\mathcal N}(u_g\alpha_t(w_k)u_h)\Vert_{2,\Tr} = 0.
\end{equation}

For all $k$, we have $w_k \in \cA \subset \rL^2(\cM, \Tr) = \bigoplus_{n \in \N} \mathcal H_n$ so that we can write $w_k = \sum_{n \in \N} w_{k,n}$, with $w_{k,n} \in \mathcal H_n$. Recall that $\Tr(\alpha_t(w_k)w_k^\ast) = \sum_{n \in \N} \left(\frac{\sin(\pi t)}{\pi t} \right)^{2n}\Vert w_{k,n} \Vert_{2,\Tr}^2$.
Thus the fact that $\lim_k \Tr(\alpha_t(w_k)w_k^*) = 0$ implies that for all $n \geq 0$, $\lim_k \Vert w_{k,n} \Vert_{2,\Tr} = 0$.

Fix $g,h \in \F_2$ and $\varepsilon > 0$. Note that for $n \geq 1$, $E_{\mathcal N}(u_g\alpha_t(w_{k,n})u_h) \in \mathcal K_n$, so that all these terms are pairwise orthogonal. They are also all orthogonal to $E_{\mathcal N}(u_g\alpha_t(w_{k,0})u_h)$, which belongs to $\mathcal K_0$. Thus
\begin{align*}
\Vert E_{\mathcal N}(u_g\alpha_t(w_k)u_h)\Vert_{2,\Tr}^2 & = \sum_{n \geq 0} \Vert E_{\mathcal N}(u_g\alpha_t(w_{k,n})u_h)\Vert_{2,\Tr}^2\\
& \leq \sum_{n \geq 0} c_n^2\Vert w_{k,n} \Vert_{2,\Tr}^2
\end{align*}
where $c_n$ is defined in Lemma \ref{cn}. Observe that $c_n \leq 1$ for all $n \in \N$.

Lemma \ref{cn} implies that there exists $n_0 \geq 0$ such that for all $n > n_0$, $c_n^2 < \varepsilon/2$. Then we can find $k_0$ such that for all $k \geq k_0$, and all $n \leq n_0$, $\Vert w_{k,n} \Vert_{2,\Tr}^2 < \varepsilon/2(n_0 +1)$. So we get that for all $k \geq k_0$,
$$\Vert E_{\mathcal N}(u_g\alpha_t(w_k)u_h)\Vert_{2,\Tr}^2 \leq \sum_{n = 0}^{n_0} \Vert w_{k,n} \Vert_{2,\Tr}^2 + \frac{\varepsilon}{2} \sum_{n \geq n_0} \Vert w_{k,n} \Vert_{2,\Tr}^2 \leq \sum_{n = 0}^{n_0} \Vert w_{k,n} \Vert_{2,\Tr}^2 + \frac{\varepsilon}{2} \Vert w_{k} \Vert_{2,\Tr}^2 \leq \varepsilon.$$
This shows $(\ref{absurd})$ and finishes the proof of Theorem \ref{dichotomy-intertwining}.
\end{proof}

\section{Relative amenability inside semifinite AFP von Neumann algebras}

Let $\mathcal M = \mathcal M_1 \ast_{\mathcal B} \mathcal M_2$ be a semifinite amalgamated free product von Neumann algebra with semifinite faithful normal trace $\Tr$. Recall that $\widetilde{\mathcal M} = \mathcal M \ast_{\mathcal B} (\mathcal B \ovt \rL(\F_2))$, $\mathcal N = \bigvee \{ u_g \mathcal M u_g^* : g \in \F_2\} \subset \widetilde{\mathcal M}$ and observe that $\widetilde{\mathcal M} = \mathcal N \rtimes \F_2$. We denote by $\alpha : \R \to \Aut(\widetilde{\mathcal M})$ the malleable deformation from Section \ref{section-malleable}.

The main result of this section is the following strengthening of Ioana's result \cite[Theorem 4.1]{Io12a} in the framework of semifinite amalgamated free product von Neumann algebras over an {\em amenable} subalgebra.

\begin{theo}\label{dichotomy-amenability}
Let $\mathcal M = \mathcal M_1 \ast_{\mathcal B} \mathcal M_2$ be a semifinite amalgamated free product von Neumann algebra with semifinite faithful normal trace $\Tr_{\mathcal M}$. Assume that $\mathcal B$ is amenable. Let $q \in \Proj(\mathcal B)$ such that $q \mathcal M_1 q \neq q \mathcal B q \neq q \mathcal M_2 q$ and $t \in (0, 1)$ such that $\alpha_t(q \mathcal M q)$ is amenable relative to $q \mathcal N q$ inside $q \widetilde{\mathcal M} q$.

Then for all $i \in \{ 1, 2 \}$, there exists a nonzero projection $z_i \in \mathcal Z(\mathcal M_i)$ such that $\mathcal M_i z_i$ is amenable.
\end{theo}

Let $\Tr_{\widetilde {\mathcal M}}$ be the semifinite faithful normal trace on $\widetilde {\mathcal M} = \mathcal M \ast_{\mathcal B} (\mathcal B \ovt \rL(\F_2))$. Consider the basic construction $\langle q\widetilde{\mathcal M} q, e_{q\mathcal N q} \rangle$ associated with the inclusion of tracial von Neumann algebras $q\mathcal N q \subset q\widetilde{\mathcal M}q$. 

We denote by $\tau = \frac{1}{\Tr_{\widetilde{\mathcal M}}(q)}\Tr_{\widetilde{\mathcal M}}(q \cdot q)$ the faithful normal tracial state on $q \widetilde{\mathcal M}q$ and by $\|\cdot\|_2$ the $\rL^2$-norm on $q\widetilde{\mathcal M}q$ associated with $\tau$. We then simply denote by $\Tr$ the canonical semifinite faithful normal trace on $\langle q\widetilde{\mathcal M}q, e_{q \mathcal N q}\rangle$ given by $\Tr(a e_{q \mathcal N q} b) = \tau(ab)$ for all $a, b \in q\widetilde{\mathcal M}q$. Observe that $q \widetilde{\mathcal M} q = q \mathcal N q \rtimes \F_2$. Following \cite[Section 4]{Io12a}, we define the $q\mathcal M_1q$-$q\mathcal M_1q$-bimodule
\[
  \mathcal H_1
  =
  \bigoplus_{g \in \F_2} \mathrm{L}^2(q \mathcal M_1 q) \, u_g e_{q\mathcal Nq} u_g^*
  \subset \mathrm{L}^2(\langle q \widetilde{\mathcal M} q, e_{q \mathcal Nq} \rangle).
\]
Denote by $\mathcal H = \mathrm{L}^2(\langle q \widetilde{\mathcal M} q, e_{q\mathcal Nq} \rangle, \Tr) \ominus \mathcal H_1$.

\begin{lem}
\label{lem:weak-containment}
As  $q\mathcal M_1 q$-$q\mathcal M_1 q$-bimodules, we have that $\mathcal H \subset_{\weak} \mathrm{L}^2(q\mathcal M_1 q) \otimes \mathrm{L}^2(q\mathcal M_1 q)$.
\end{lem}
\begin{proof}
The proof goes along the same lines as \cite[Lemma 4.2]{Io12a}.  First observe that since $q\widetilde{\mathcal M}q = q\mathcal Nq \rtimes \F_2$, we have
\[
  \mathrm{L}^2(\langle q \mathcal M q, e_{q \mathcal Nq} \rangle )
  \cong
  \bigoplus_{g,h \in \F_2} \mathrm{L}^2(q \mathcal Nq) \, u_g e_{q \mathcal Nq} u_h.
\]
So it suffices to prove that for all $g,h \in \F_2$ such that $h \neq g^{-1}$, as $q \mathcal M_1 q$-$q \mathcal M_1 q$-bimodules, we have
\begin{align*}
\left(\mathrm{L}^2(q\mathcal Nq) \ominus \mathrm{L}^2(q\mathcal M_1 q) \right)u_g e_{q\mathcal Nq} u_g^* & \subset_{\weak} \mathrm{L}^2(q\mathcal M_1 q) \otimes \mathrm{L}^2(q\mathcal M_1 q) \\
  \mathrm{L}^2(q\mathcal Nq) \, u_g e_{q\mathcal Nq} u_h & \subset_{\weak} \mathrm{L}^2(q\mathcal M_1 q) \ot \mathrm{L}^2(q\mathcal M_1 q).
\end{align*}

Denote by $\mathrm{L}^2(q\mathcal Nq)^g$ the $q\mathcal M_1 q$-$q\mathcal M_1 q$-bimodule $\mathrm{L}^2(q\mathcal Nq)$ with left and right action given by $x \cdot \xi \cdot y = x \xi u_gy u_g^*$ for all $x, y \in q \mathcal M_1 q$ and all $\xi \in \mathrm{L}^2(q \mathcal N q)$. Likewise, define the $\mathcal M_1$-$\mathcal M_1$-bimodule $\rL^2(\mathcal N)^g$. As $q \mathcal M_1 q$-$q \mathcal M_1 q$-bimodules, we have
\begin{align*}
\bigoplus_{g \in \F_2}
  \left(\mathrm{L}^2(q\mathcal Nq) \ominus \mathrm{L}^2(q \mathcal M_1 q) \right)u_g e_{q \mathcal Nq} u_g^*
 & \cong
  \bigoplus_{i = 1}^\infty \left( \mathrm{L}^2(q\mathcal Nq) \ominus \mathrm{L}^2(q \mathcal M_1 q)\right) \\
\bigoplus_{g,h \in \F_2, h \neq g^{-1}} \mathrm{L}^2(q\mathcal Nq) \, u_ge_{q\mathcal Nq} u_h
 & \cong
  \bigoplus_{i = 1}^\infty \bigoplus_{g \in \F_2 \setminus \{e\}} \mathrm{L}^2(q\mathcal Nq)^{g}.
\end{align*}

Put $\mathcal P = \left(\bigcup_{h \in \F_2 \setminus \{e\}} u_h \mathcal M u_h^* \cup \mathcal M_2\right)\dpr$ and $\mathcal P_g =  \left(\bigcup_{h \in \F_2 \setminus \{e,g\}} u_h \mathcal M u_h^* \cup \mathcal M_2 \cup u_g \mathcal M_2 u_g^*\right)\dpr$ for all $g \in \F_2$.  Then we have
$$\mathcal N \cong \mathcal M_1 *_{\mathcal B} \mathcal P \cong \mathcal M_1 *_{\mathcal B} u_g\mathcal M_1u_g^* *_{\mathcal B} \mathcal P_g, \forall g \in \F_2 \setminus \{e\}.$$  
Using \cite[Section 2]{ueda-pacific}, there are $\mathcal B$-$\mathcal B$-bimodules $\mathcal L$ and $\mathcal L_g$ for $g \in \F_2 \setminus \{e\}$, such that as $\mathcal M_1$-$\mathcal M_1$-bimodules, we have
\begin{align*}
\mathrm{L}^2(\mathcal N) \ominus \rL(\mathcal M_1) & \cong \mathrm{L}^2(\mathcal M_1) \otimes_{\mathcal B} \mathcal L \otimes_{\mathcal B} \mathrm{L}^2(\mathcal M_1) \\
\rL^2(\mathcal N)^g & \cong \mathrm{L}^2(\mathcal M_1) \otimes_{\mathcal B} \mathcal L_g \otimes_{\mathcal B} \mathrm{L}^2(\mathcal M_1).
\end{align*}

Since $\mathcal B$ is amenable, we have that $\mathrm{L}^2(\mathcal B) \subset_{\weak} \mathrm{L}^2(\mathcal B) \otimes \mathrm{L}^2(\mathcal B)$ as $\mathcal B$-$\mathcal B$-bimodules.  Using \cite[Lemma 1.7]{AD93}, we obtain that, as $q\mathcal M_1 q$-$q \mathcal M_1 q$-bimodules,
\begin{align*}
  \mathrm{L}^2(q\mathcal Nq) \ominus \mathrm{L}^2(q \mathcal M_1 q)
  & \cong
  q \left(\mathrm{L}^2(\mathcal M_1) \ot_{\mathcal B} \mathcal L \ot_{\mathcal B} \mathrm{L}^2(\mathcal M_1) \right)q \\
  & \subset_{\weak} q \left(\mathrm{L}^2(\mathcal M_1) \otimes \mathcal L \otimes \mathrm{L}^2(\mathcal M_1) \right)q \\
  & \subset_{\weak}
  q \left(\mathrm{L}^2(\mathcal M_1) \otimes \mathrm{L}^2(\mathcal M_1) \right)q.
\end{align*}
Since $q \left(\mathrm{L}^2(\mathcal M_1) \otimes \mathrm{L}^2(\mathcal M_1) \right)q$ is isomorphic to a $q \mathcal M_1 q$-$q \mathcal M_1 q$-subbimodule of $\bigoplus_{i = 1}^\infty \mathrm{L}^2(q \mathcal M_1 q) \ot \mathrm{L}^2(q \mathcal M_1 q)$, we infer that, as $q\mathcal M_1q$-$q\mathcal M_1q$-bimodules, 
$$\mathrm{L}^2(q\mathcal Nq) \ominus \mathrm{L}^2(q \mathcal M_1 q)\subset_{\weak}\mathrm{L}^2(q\mathcal M_1 q) \otimes \mathrm{L}^2(q \mathcal M_1 q).$$ 
Similarly, for all $g \in \F_2 \setminus \{e\}$ we get that, as $q\mathcal M_1q$-$q\mathcal M_1q$-bimodules,
\begin{equation*}
 \mathrm{L}^2(q\mathcal Nq)^g \subset_{\weak}
  \mathrm{L}^2(q\mathcal M_1 q) \otimes \mathrm{L}^2(q \mathcal M_1 q).\qedhere
\end{equation*}
\end{proof}

\begin{proof}[Proof of Theorem \ref{dichotomy-amenability}]
Since $\alpha_t(q \mathcal M q)$ is amenable relative to $q\mathcal Nq$ inside $q \widetilde{\mathcal M} q$, we find a net of vectors $\xi_n \in \mathrm{L}^2(\langle q \widetilde{\mathcal M} q, e_{q\mathcal Nq} \rangle, \Tr)$ for $n \in I$, such that
\begin{itemize}
\item $\langle x\xi_n \, | \, \xi_n \rangle_{\Tr} \rightarrow \tau(x)$ for all $x \in q \widetilde{\mathcal M} q$, and
\item $\| x \xi_n - \xi_n x \|_{2, \Tr} \rightarrow 0$ for all $x \in \alpha_t(q \mathcal M q)$.
\end{itemize}
Observe that using the proof of \cite[Theorem 2.1]{OP07} we may assume that $\xi_n \geq 0$ so that $\langle x\xi_n \, | \, \xi_n \rangle_{\Tr} = \Tr(x \xi_n^2) = \langle \xi_n x \, | \, \xi_n \rangle_{\Tr}$ for all $x \in q \widetilde{\mathcal M} q$ and all $n \in I$. Since $\|\xi_n\|_{2, \Tr} \to 1$, we may further assume that $\|\xi_n\|_{2, \Tr} = 1$ for all $n \in I$.

By contradiction, assume that for some $i \in \{1,2\}$, $q \mathcal M_iq$ has no amenable direct summand.  Without loss of generality, we may assume that $i = 1$.  Denote by $P_{\mathcal H_1}: \mathrm{L}^2(\langle q \widetilde{\mathcal M} q, e_{q\mathcal Nq} \rangle)\to \mathcal H_1$ the orthogonal projection. Observe that $P_{\mathcal H_1}$ is the orthogonal projection corresponding to the unique trace preserving faithful normal conditional expectation $E_{\mathcal Q} : q\widetilde{\mathcal M}q \to \mathcal Q$ onto the von Neumann subalgebra $\mathcal Q = \bigvee \{q \mathcal M_1 q, u_g e_{q \mathcal N q} u_g^* : g \in \mathbf F_2\} \subset q \widetilde{\mathcal M}q$. We claim that $\lim_n \|u_1^{t*} \xi_n u_1^t - P_{\mathcal H_1}(u_1^{t*}\xi_n u_1^t)\|_{2, \Tr} = 0$.  If this is not the case, let $\zeta_n = (1 - P_{\mathcal H_1})(u_1^{t*} \xi_n u_1^t) \in \mathcal H$ and observe that $\limsup_n \|\zeta_n\|_{2, \Tr} > 0$. Arguing as in the proof of \cite[Lemma 2.3]{Io12a}, we may further assume that $\liminf_n \|\zeta_n\|_{2, \Tr} > 0$.

Then $\zeta_n \in \mathcal H$ is a net of vectors which satisfies the following conditions:
\begin{itemize}
\item $\liminf_n \|\zeta_n\|_{2, \Tr} > 0$;
\item $\limsup_n \|x \zeta_n\|_{2, \Tr} \leq \|x\|_2$ for all $x \in q \mathcal M_1 q$;
\item $\lim_n \| y \zeta_n - \zeta_n y\|_{2, \Tr} = 0$ for all $y \in q \mathcal M_1 q$.
\end{itemize}
Since as $q\mathcal M_1q$-$q\mathcal M_1q$-bimodules, we have that $\mathcal H\subset_{\weak} \mathrm{L}^2(q \mathcal M_1 q) \ot \mathrm{L}^2(q \mathcal M_1 q)$ by Lemma \ref{lem:weak-containment}, it follows that $q \mathcal M_1 q$ has an amenable direct summand by Connes's result \cite{connes76}. This contradicts our assumption and we have shown that $\lim_n \| \xi_n  - u_1^t P_{\mathcal H_1}(u_1^{t*}\xi_n u_1^t) u_1^{t*}\|_{2, \Tr} = \lim_n \|u_1^{t*} \xi_n u_1^t - P_{\mathcal H_1}(u_1^{t*}\xi_n u_1^t)\|_{2, \Tr} = 0$.

Put $\mathcal L_1 = u_1^t \mathcal H_1 u_1^{t*}$ and denote by $P_{\mathcal L_1} : \mathrm{L}^2(\langle q \widetilde{\mathcal M} q, e_{q\mathcal Nq} \rangle)\to \mathcal L_1$ the orthogonal projection. Put $\eta_n = P_{\mathcal L_1}(\xi_n)$ and observe that $\eta_n \in \mathcal L_1$ and $\eta_n \geq 0$. We moreover have $\lim_n \|\xi_n - \eta_n\|_{2, \Tr} = 0$.  So $\eta_n \in \mathcal L_1$ is a net of vectors which satisfy 
\begin{enumerate}
\item [$(\ast)$] $\langle x\eta_n \, | \, \eta_n \rangle_{\Tr} = \langle \eta_n x\, | \, \eta_n \rangle_{\Tr} \rightarrow \tau(x)$ for all $x \in q \widetilde{\mathcal M} q$, and
\item [$(\ast \ast)$] $\| x \eta_n - \eta_n x \|_{2, \Tr} \rightarrow 0$ for all $x \in \alpha_t(q \mathcal M q)$.
\end{enumerate}

We have $\eta_n = \sum_{g \in \F_2}u_1^t x_{n,g} u_g e_{q\mathcal Nq} u_g^* u_1^{t*}$ with $x_{n,g} \in \mathrm{L}^2(q \mathcal M_1q)$.  Since $\eta_n = \eta_n^*$ for all $n \in I$, we may assume that $x_{n, g} = x_{n, g}^*$ for all $n \in I$ and all $g \in \F_2$. Next, we claim that we may further assume that $x_{n, g} \in q \mathcal M_1 q$ with $x_{n, g} = x_{n, g}^*$ for all $n \in I$ and all $g \in \F_2$. 

To do so, define the set $J$ of triples $j = (X, Y, \varepsilon)$ where $X \subset \Ball(q\widetilde{\mathcal M}q)$, $Y \subset \Ball(\alpha_t(q \mathcal M q))$ are finite subsets and $\varepsilon > 0$. We make $J$ a directed set by putting $(X, Y, \varepsilon) \leq (X' , Y' , \varepsilon')$ if and only if $X \subset X'$, $Y \subset Y'$ and $\varepsilon' \leq \varepsilon$. Let $j = (X, Y, \varepsilon) \in J$. There exists $n \in I$ such that $|\langle x\eta_n \, | \, \eta_n \rangle_{\Tr} - \tau(x) | \leq \varepsilon/2$ and $\|y \eta_n - \eta_n y\|_{2, \Tr} \leq \varepsilon/2$ for all $x \in X$ and all $y \in Y$. Let $\upsilon \in \ell^2(\F_2)_+$ such that $\|\upsilon\|_{\ell^2(\F_2)} = 1$. For each $g \in \F_2$, choose $y_{j, g} \in q \mathcal M_1 q$ such that $y_{j, g} = y_{j, g}^*$ and $\|x_{n, g} - y_{j, g}\|_2 \leq \upsilon(g) \, \varepsilon/4$. Put $ \eta_j' = \sum_{g \in \F_2}u_1^t y_{j,g} u_g e_{q\mathcal Nq} u_g^* u_1^{t*} \in \mathcal L_1$ and observe that $ \eta_j' =  \eta_j'^*$ and $\|\eta_n - \eta_j'\|_{2, \Tr} \leq \varepsilon/4$. We get $|\langle x\eta_j' \, | \, \eta_j' \rangle_{\Tr} - \tau(x) | \leq \varepsilon + \varepsilon^2/16$ and $\|y \eta_j' - \eta_j' y\|_{2, \Tr} \leq \varepsilon$ for all $x \in X$ and all $y \in Y$. Then the net $(\eta_j')_{j \in J}$ clearly satisfies Conditions $(\ast)$ and $(\ast \ast)$ above. This finishes the proof of the claim.

Fix any $y \in q\mathcal M_2q \ominus q\mathcal B q$ satisfying $\|y\|_2 = 1$.  Then we have
\[
  \langle \alpha_t(y) \eta_n \, | \, \eta_n \alpha_t(y) \rangle_{\Tr}
  \rightarrow
  1.
\]
Expanding $\alpha_t(y)$ and $\eta_n$, we obtain
\begin{align*}
\langle \alpha_t(y) \eta_n \, | \, \eta_n \alpha_t(y) \rangle_{\Tr} 
  & =
  \sum_{g, h \in \F_2}
  {
    \left \langle u_2^t y u_2^{t *} \, u_1^t x_{n,g} u_g e_{q\mathcal Nq} u_g^* u_1^{t*} \, | \,
    u_1^t x_{n,h} u_h e_{q \mathcal Nq} u_h^* u_1^{t*} \, u_2^t y u_2^{t *} \right \rangle_{\Tr}
  } \\
  & =
  \sum_{g, h \in \F_2}
  {
    \left \langle u_h^* x_{n,h}^* u_1^{t*} u_2^t y u_2^{t *}  u_1^t x_{n,g} u_g \, e_{q\mathcal Nq} \, | \,
    e_{q\mathcal Nq} \, u_h^* u_1^{t*}  u_2^t y u_2^{t *} u_1^t u_g \right \rangle_{\Tr}
  } \\
  & =
  \sum_{g,h \in \F_2}
  {
    \tau \left (E_{q\mathcal Nq}(u_g^* u_1^{t*} u_2^t y^* u_2^{t *} u_1^t u_h) \, E_{q\mathcal Nq}(u_h^* x_{n,h}^* u_1^{t *} u_2^t y u_2^{t *} u_1^t x_{n,g} u_g) \right)
  }.
\end{align*}
Recall from Section \ref{section-malleable} the definition of the Hilbert spaces $\mathcal K_k$ for  $k \in \N$ and denote by $b_{n,g} = E_{q \mathcal B q}(x_{n,g})$.  Since we have 
$$
E_{q\mathcal Nq}\left(u_g^* u_1^{t*} u_2^t y^* u_2^{t *} u_1^{t} u_h \right) \in \mathcal K_1,
$$
$$
E_{q\mathcal Nq}\left(u_h^* (x_{n,h} - b_{n,g})^* u_1^{t *} u_2^t y u_2^{t *} u_1^t b_{n,g} u_g \right) \text{ and } E_{q\mathcal Nq}\left(u_h^* b_{n,g}^* u_1^{t *} u_2^t y u_2^{t *} u_1^t (x_{n, g} - b_{n,g}) u_g\right)  \in \mathcal K_2,
$$
$$ E_{q\mathcal Nq}\left(u_h^* (x_{n,h} - b_{n,g})^* u_1^{t *} u_2^t y u_2^{t *} u_1^t (x_{n,g} - b_{n,g}) u_g\right) \in \mathcal K_3,
$$
we get
\begin{align*}
  \langle \alpha_t(y) \eta_n \, | \, \eta_n \alpha_t(y) \rangle_{\Tr} 
  & =
  \sum_{g,h \in \F_2}
  {
    \tau \left (E_{q\mathcal Nq}(u_g^* u_1^{t*} u_2^t y^* u_2^{t *} u_1^{t} u_h) \, E_{q\mathcal Nq}(u_h^* b_{n,h}^* u_1^{t *} u_2^t y u_2^{t *} u_1^t b_{n,g} u_g) \right)
  } \\
  & =
    \sum_{g,h \in \F_2}
  {
    \tau \left (E_{q\mathcal Nq}(u_g^* u_1^{t*} u_2^t y^* u_2^{t *} u_1^{t} u_h) \, E_{q \mathcal Nq}(u_h^* u_1^{t *} u_2^t (b_{n,h}^*  y b_{n,g}) u_2^{t *} u_1^t u_g) \right)
  }.
\end{align*}

As in the proof of Theorem \ref{dichotomy-intertwining}, for $i \in \{ 1, 2 \}$, put $G_1 = \langle a \rangle$ and $G_2 = \langle b \rangle$ so that $u_1 = u_a$ and $u_2 = u_b$. Denote by $(\beta_i(g))_{g \in G_i}$ the Fourier coefficients of $u_i^t$. For $g, h \in \F_2$, define the isometry $W_{g, h} : \rL^2(\mathcal M_2) \ominus \rL^2(\mathcal B) \to \rL^2(\widetilde{\mathcal M})$ by $W_{g, h}(x) = u_g x u_h^*$ for $x \in \mathcal M_2 \ominus \mathcal B$ such that $\Tr_{\mathcal M}(x^* x) < \infty$. Thanks to Lemma \ref{orthogonalranges}, the isometries $W_{g, h}$ have pairwise orthogonal ranges when $(g, h)$ are pairwise distinct.  For all $z \in q \mathcal M_2 q \ominus q \mathcal B q$ and all $g,h \in \F_2$, using calculation $(\ref{esp})$, we obtain
\begin{align*}
  E_{q\mathcal Nq}(u_h^* u_1^{t *} u_2^t z u_2^{t *} u_1^t u_g) 
  & =
  \sum_{r,r' \in G_1 ,s,s' \in G_2}
  {
    \beta_1(r)
    \beta_2(s)
    \beta_2(s')
    \beta_1(r')
    \, E_{q \mathcal Nq} ( W_{h^{-1} r^{-1}s, g^{-1} r'^{-1}s'}(z) )
   } \\
  & =
  \sum_{\substack{ r,r' \in G_1, s,s' \in G_2 \\ h^{-1} r^{-1} s s'^{-1} r' g = 1 }}
  {
    \beta_1(r)
    \beta_2(s)
    \beta_2(s')
    \beta_1(r')
    \,  W_{h^{-1} r^{-1}s, g^{-1} r'^{-1} s'}(z)
  }.
\end{align*}
Using the facts that $G_1 \cap G_2 = \{e\}$ and that the isometries $W_{g', h'}$ have pairwise orthogonal ranges when $(g', h')$ are pairwise distinct, we get 
\begin{align*} 
& \quad \tau \left (
    E_{q\mathcal Nq}(u_g^* u_1^{t*} u_2^t y^* u_2^{t *} u_1^{t} u_h) \, E_{q\mathcal Nq}(u_h^* u_1^{t *} u_2^t (b_{n,h}^*  y b_{n,g}) u_2^{t *} u_1^t u_g)
  \right) \\
   &= 
  \sum_{\substack{ r,r' \in G, s,s' \in G_2 \\ r s s' r'  = hg^{-1} }}
  {
    \beta_1(r^{-1}) ^2
    \beta_2(s) ^2
    \beta_2(s'^{-1}) ^2
    \beta_1(r') ^2
    \, \tau (y^* b_{n,h}^*y b_{n,g})
  }.
\end{align*}

For $i \in \{1,2\}$, define $\mu_i \in \Prob(\F_2)$ by $\mu_i(g) = \beta_i(g)^2$ if $g \in G_i$ and $\mu_i(g) = 0$ otherwise. Likewise, define $\check \mu_i \in \Prob(\F_2)$ by $\check \mu_i(g) = \mu_i(g^{-1})$ for all $g \in \F_2$. Put $\mu = \check \mu_1 \ast  \mu_2 \ast \check \mu_2 \ast \mu_1$. Since $y \in q\mathcal M_2q \ominus q\mathcal Bq$ and $x_{n, g} \in q\mathcal M_1q$, we obtain that
\begin{align*}
  \tau \left (
    E_{q\mathcal Nq}(u_g^* u_1^{t*} u_2^t y^* u_2^{t *} u_1^{t} u_h) \, E_{q\mathcal Nq}(u_h^* u_1^{t *} u_2^t (b_{n,h}^*  y b_{n,g}) u_2^{t *} u_1^t u_g)
  \right)
  &= 
  \mu(hg^{-1})\,\tau(y^* b_{n,h}^*y b_{n,g}) \\
  &= \mu(hg^{-1})\,\tau(y^* x_{n,h}^*y x_{n,g}).
\end{align*}

Summing over all $g, h \in \F_2$ and using Cauchy-Schwarz inequality, we get
\begin{align*}
  |\langle \alpha_t(y) \eta_n \, | \, \eta_n \alpha_t(y) \rangle_{\Tr}|
  & =
  \left |
    \sum_{g,h \in \F_2}{ \mu(hg^{-1})\tau(y^* x_{n,h}^*y x_{n,g}) }
  \right | \\
  & =
  \left |
    \sum_{g,h \in \F_2}{ \mu(g)\tau(y^* x_{n,h}^*y x_{n,g^{-1} h}) }
  \right | \\
  & \leq
  \sum_{g,h \in \F_2} { \mu(g) \|x_{n,h} y\|_2 \, \|y x_{n,g^{-1}h}\|_2 } \\
  & \leq
  \sum_{g \in \F_2} { \mu(g) \langle \zeta_n \, | \, \lambda_g(\zeta'_n) \rangle_{\ell^2(\F_2)}},
\end{align*}
where $\zeta_n = \sum_{h \in \F_2} \|x_{n,h} y\|_2 \, \delta_h$ and $\zeta'_n = \sum_{h \in \F_2} \|y x_{n,h}\|_2 \, \delta_h$.  Since we moreover have $u_1^{t*}\eta_n u_1^t = \sum_{g \in \F_2}  u_g e_{q \mathcal Nq} u_g^* \, x_{n,g}$, we get
\[
  \| u_1^{t *} \eta_n u_1^t y \|^2_{2, \Tr}
  =
  \sum_{g \in \F_2} \| u_g e_{q \mathcal Nq} u_g^* \, x_{n, g} y\|_{2, \Tr}^2 =  \sum_{g \in \F_2} \|x_{n,g} y\|_2^2 = \|\zeta_n\|_{\ell^2(\F_2)}^2.
\]
Likewise we have $\|\zeta'_n\|_{\ell^2(\F_2)} = \|y u_1^{t *} \eta_n u_1^t\|_{2, \Tr}$.

Denote by $T: \ell^2(\F_2) \rightarrow \ell^2(\F_2)$ the Markov operator defined by $T = \sum_{g \in \F_2} \mu(g) \lambda_g$.  Since the support of $\mu$ generates $\F_2$ and $\mu(e) > 0$ (see the proof of \cite[Lemma 3.4, Claim]{Io12a}), Kesten's criterion for amenability \cite{kesten} yields $\| T \|_\infty < 1$.  This gives
\begin{align*}
  | \langle \alpha_t(y)\eta_n \, | \, \eta_n \alpha_t(y) \rangle_{\Tr} |
  & \leq \langle \zeta_n \, | \, T \zeta_n' \rangle_{\ell^2(\F_2)} \\
  & \leq \|T\|_\infty \, \|\zeta_n \|_{\ell^2(\F_2)} \, \|\zeta_n' \|_{\ell^2(\F_2)} \\
  & = \| T \|_\infty \, \|u_1^{t *} \eta_n u_1^t y\|_{2, \Tr} \, \|y u_1^{t *} \eta_n u_1^t\|_{2, \Tr} \\
  & = \| T \|_\infty \, \| \eta_n u_1^t y\|_{2, \Tr} \, \|y u_1^{t*}\eta_n\|_{2, \Tr}.
\end{align*}
Since $\eta_n = \eta_n^*$,  we obtain 
\[
  \|\eta_n u_1^t y\|_{2, \Tr} \, \|y u_1^{t *} \eta_n\|_{2, \Tr}
  \rightarrow
  \|u_1^t y\|_2 \, \|y u_1^{t *}\|_2 
  = 
  \|y\|_2^2
  =
  1,
\]
hence $\limsup_n  | \langle \alpha_t(y)\eta_n \, | \, \eta_n \alpha_t(y) \rangle_{\Tr} | \leq \|T\|_\infty < 1$. This however contradicts the fact that
\[
  | \langle \alpha_t(y)\eta_n \, | \, \eta_n \alpha_t(y) \rangle_{\Tr} |
  \rightarrow
  1
\]
and hence our assumption that $q \mathcal M_1q$ had no amenable direct summand. Thus for all $i \in \{1, 2\}$, $q \mathcal M_i q$ has an amenable direct summand and so does $\mathcal M_i$. This finishes the proof of Theorem~\ref{dichotomy-amenability}.
\end{proof}

A combination of the proof  of the above Theorem \ref{dichotomy-amenability}  and the one of \cite[Theorem 4.1]{Io12a} shows that ``or" can be replaced with ``and" in Ioana's result \cite[Theorem 4.1]{Io12a}.

\begin{theo}
Let $M = M_1 \ast_B M_2$ be a tracial amalgamated free product von Neumann algebra. Assume that $M_1 \neq B \neq M_2$. Put $\widetilde M = M \ast_B (B \ovt  \rL(\F_2)) = N \rtimes \F_2$ where $N = \bigvee \{u_g M u_g^* : g \in \rL(\F_2)\}$. Let $t \in (0, 1)$ such that $\alpha_t(M)$ is amenable relative to $N$. 

Then for all $i \in \{1, 2\}$, there exists a nonzero projection $z_i \in \mathcal Z(M_i)$ such that $M_i z_i$ is amenable relative to $B$ inside $M$.
\end{theo}

\section{Proofs of Theorems \ref{thmA} and \ref{thmB}}

\subsection{A general intermediate result}

Theorems \ref{thmA} and \ref{thmB} will be derived from the following very general result regarding Cartan subalgebras inside semifinite amalgamated free product von Neumann algebras.

\begin{theo}\label{general-semifinite}
Let $\mathcal M = \mathcal M_1 \ast_{\mathcal B} \mathcal M_2$ be a semifinite amalgamated free product von Neumann algebra with semifinite faithful normal trace $\Tr$. Assume that $\mathcal B$ is amenable, $\mathcal M_1$ has no amenable direct summand and for all nonzero projections $e \in \mathcal B$, we have $e \mathcal B e \neq e \mathcal M_2 e$.

Let $p \in \Proj(\mathcal B)$ and $\mathcal A \subset p \mathcal M p$ any regular amenable von Neumann subalgebra. Then there exists $q \in \Proj(\mathcal B)$ such that $\mathcal A  \preceq_{\mathcal M} q \mathcal B q$.
\end{theo}

\begin{proof}
Put $\widetilde{\mathcal M} = \mathcal M \ast_{\mathcal B} (\mathcal B \ovt \rL(\F_2))$ and regard $p \widetilde{\mathcal M} p$ as the tracial crossed product von Neumann algebra $p \widetilde{\mathcal M} p = p\mathcal Np \rtimes \F_2$ with $\mathcal N = \bigvee \{u_g \mathcal M u_g : g \in \F_2\}$. We denote by $(\alpha_t)$ the malleable deformation from Section \ref{section-malleable}. Applying Popa-Vaes's dichotomy result \cite[Theorem 1.6]{PV11} to the inclusion $\alpha_t(\mathcal A) \subset p \widetilde{\mathcal  M} p$ for $t \in (0, 1)$, we get that at least one of the following holds true:
\begin{enumerate}
\item Either $\alpha_t(\mathcal A) \preceq_{p \widetilde{\mathcal M} p} p \mathcal N p$.
\item Or $\alpha_t(p \mathcal M p)$ is amenable relative to $p \mathcal N p$ inside $p \widetilde{\mathcal M} p$.
\end{enumerate}

Since $\mathcal M_1$ has no amenable direct summand, case (2) cannot hold by Theorem \ref{dichotomy-amenability}. It remains to show that case (1) leads to the conclusion of the theorem.

In case (1), using Lemma \ref{intertwining-general-HR} and Theorem \ref{dichotomy-intertwining}, we get that either there exists $q \in \Proj(\mathcal B)$ such that $\mathcal A \preceq_{\mathcal M} q \mathcal B q$ or there exist $i \in \{1, 2\}$ and $q_i \in \Proj(\mathcal M_i)$ such that $p \mathcal M p \preceq_{\mathcal M} q_i \mathcal M_i q_i$. Since the latter case is impossible by Proposition \ref{intertwining-afp}, we get $\mathcal A \preceq_{\mathcal M} q\mathcal Bq$ for some $q \in \Proj(\mathcal B)$.
\end{proof}

\subsection{Proof of Theorem \ref{thmA}}

We first need to prove the following well-known result.

\begin{lem}\label{nontrivial-free-product}
Let $M$ be any von Neumann algebra such that $M \neq \C$ and $\varphi$ any faithful normal state on $M$. Realize the continuous core $\core(M) = M \rtimes_\varphi \R$. Then for every nonzero projection $p \in \rL(\R)$, we have $\rL(\R) p \neq p\core(M)p$.
\end{lem}

\begin{proof}
There are two cases to consider.

{\bf Case (1): assume that $M^\varphi \neq \C$.} Choose $r \in M^\varphi$ a projection such that $r \neq 0, 1$. Observe that $x = \varphi(1 - r) \, r - \varphi(r) \, (1 - r) \in M^\varphi$ is invertible and $\varphi(x) = 0$. Then for every nonzero projection $p \in \rL(\R)$, we have $x p \neq 0$ and $E_{\rL(\R) p}(x p) = \varphi(x) p = 0$. This proves that $\rL(\R)p \neq pMp$.

{\bf Case (2): assume that $M^\varphi = \C$.} Since $\mathcal Z(M) \subset \mathcal Z(M^\varphi)$, it follows that $M$ is a factor. If $M$ is of type ${\rm III}$, it follows from Connes's classification of type ${\rm III}$ factors \cite{connes73} that $M$ is necessarily of type ${\rm III_1}$. In that case, $\core(M)$ is a type ${\rm II_\infty}$ factor and thus $\rL(\R)p \neq p \core(M) p$ for every nonzero projection $p \in \rL(\R)$. If $M$ is a semifinite factor with semifinite faithful normal trace $\Tr$, there exists $b \in \rL^1(M, \Tr)_+$ such that $\varphi = \Tr(b \, \cdot)$ and $\|b\|_{1, \Tr} = 1$. Let $q \in M$ be a nonzero spectral projection of $b$. Since 
$$\varphi(qx) = \Tr(b q x) = \Tr(q b x) = \Tr(b x q) = \varphi(xq)$$
for all $x \in M$, we get $q \in M^\varphi$ and so $q = 1$. This shows that $b = 1$ and $\Tr = \varphi$ is a finite trace on $M$. Hence $M = M^\varphi = \C$, which is a contradiction.
\end{proof}

\begin{proof}[Proof of Theorem \ref{thmA}]
By \cite[Theorem 4.1]{ueda-advances}, we know that there exists a nonzero projection $z \in \mathcal Z(M)$ such that $M z$ is a full factor and $M(1 - z)$ is a purely atomic von Neumann algebra. In particular, $M$ is not amenable.

In the case when both $M_1$ and $M_2$ are amenable, \cite[Theorem 5.5]{houdayer-ricard} implies that $M$ has no Cartan subalgebra. It remains to consider the case when $M_1$ or $M_2$ is not amenable. Without loss of generality, we may assume that $M_1$ is not amenable.

By contradiction, assume that $M$ has a Cartan subalgebra. Hence, $Mz$ also has a Cartan subalgebra. Let $p \in \mathcal Z(M_1)$ be the largest nonzero projection such that $M_1 p$ has no amenable direct summand. Since $M(1 - z)$ is purely atomic, we necessarily have $p \leq z$.

By \cite[Lemma 2.2]{ueda-advances}, we have
$$(pMp, \frac{1}{\varphi(p)} \varphi (p \cdot p)) = (M_1 p, \frac{1}{\varphi_1(p)} \varphi_1(\cdot p)) \ast (pNp, \frac{1}{\varphi(p)}\varphi (p \cdot p))$$
with $N = (\C p \oplus M_1(1 - p)) \vee M_2$. Observe that $pNp \neq \C p$. Indeed let $q \in M_2$ be a projection such that $q \neq 0, 1$. Then $pqp = \varphi_2(q)p + p(q - \varphi_2(q))p \in pNp \setminus \C p$. Since $Mz$ is a factor and $p \leq z$, it follows that $pMp$ has a Cartan subalgebra by \cite[Lemma 3.5]{popa-malleable1}. 

From the previous discussion, it follows that we may assume that $M_1$ has no amenable direct summand, $M_2 \neq \C$ and $M$ has a Cartan subalgebra  $A \subset M$. Using Notation \ref{notation}, denote by $\core(A) \subset \core(M)$ the Cartan subalgebra in the continuous core $\core(M) = \core(M_1) \ast_{\rL(\R)} \core(M_2)$. Let $q \in \Proj(\rL(\R))$. Since $\core(A) \subset \core(M)$ is maximal abelian and $\Tr | \core(A)$ is semifinite, \cite[Lemma 2.1]{houdayer-vaes} shows that there exists a nonzero finite trace projection $p \in \core(A)$ and a partial isometry $v \in \core(M)$ such that $p = v^*v$ and $q = vv^*$. Observe that $v \core(A) v^* \subset q \core(M) q$ is still a Cartan subalgebra by \cite[Lemma 3.5]{popa-malleable1}. 

By Lemma \ref{intertwining-general-HR}, Proposition \ref{amenable-core}, Theorem \ref{general-semifinite} and Lemma \ref{nontrivial-free-product}, there exists $q' \in \Proj(\rL(\R))$ such that $v \core(A) v^*\preceq_{\core(M)} \rL(\R) q'$. Then Proposition \ref{intertwining-core} implies that $A \preceq_M \C$. This contradicts the fact that $A$ is diffuse and finishes the proof of Theorem \ref{thmA}.
\end{proof}

\subsection{Proof of Theorem \ref{thmB}}

\begin{proof}[Proof of Theorem \ref{thmB}]
Let $A \subset M$ be a Cartan subalgebra. Since $A, B \subset M$ are both tracial von Neumann subalgebras of $M$ with expectation, we use Notation \ref{notation}. Let  $q \in \Proj(\mathcal Z(\core(B)))$. By \cite[Lemma 2.1]{houdayer-vaes}, there exists $p \in \Proj(\core(A))$ and a partial isometry $v \in \core(M)$ such that $p = v^*v$ and $q = vv^*$. Observe that $v \core(A) v^* \subset q \core(M) q$ is still a Cartan subalgebra by \cite[Lemma 3.5]{popa-malleable1}. 

Using the assumptions, by Lemma \ref{intertwining-general-HR}, Proposition \ref{amenable-core}, \cite[Proposition 5.5]{houdayer-vaes} and Theorem \ref{general-semifinite}, there exists $q' \in \Proj(\mathcal Z(\core(B)))$ such that $v \core(A) v^*\preceq_{\core(M)} \core(B) q'$. Then Proposition \ref{intertwining-core} implies that $A \preceq_M B$.
\end{proof}

\section{Proof of Theorem \ref{thmC}} 

Let $\mathcal R$ be any countable nonsingular equivalence relation on a standard measure space $(X, \mu)$. Following \cite{feldman-moore}, denote by $m$ the measure on $\cR$ given by
$$m(\mathcal W) = \int_X |\{y \in X : (x,y) \in \mathcal W\}| \; \rd\mu(x) $$
for all measurable subsets $\mathcal W \subset \cR$. We denote by $[\mathcal R]$ the full group of $\mathcal R$, $M = \rL(\cR)$ the von Neumann algebra of $\mathcal R$ and identify $\rL^2(M) = \rL^2(\cR,m)$. For all $\psi\in [\cR]$, define $u(\psi) \in \mathcal U(M)$ whose action on $\rL^2(\cR,m)$ is given by 
$$(u(\psi)\xi)(x,y) = \left( \frac{{\rm d}( \mu \circ \psi^{-1})}{{\rm d} \mu}(x)\right)^{1/2} \xi(\psi^{-1}(x),y).$$

We view $\rL^\infty(\cR)$ as acting on $\rL^2(\cR,m)$ by multiplication operators. Note that the unitaries $u(\psi) \in \mathcal U(M)$ for $\psi \in [\cR]$ normalize $\rL^\infty(\cR)$ and that $\rL^\infty(X) \subset \rL^\infty(\cR)$, by identifying a function $F \in \rL^\infty(X)$ with the function on $\cR$ given by $(x,y) \mapsto F(x)$.

Recall from \cite[Definition 5]{connes-feldman-weiss} that $\cR$ is \emph{amenable} if there exists a norm one projection $\Phi : \rL^\infty(\mathcal R) \to \rL^\infty(X)$ satisfying
$$\Phi(u(\psi) F u(\psi)^*) = u(\psi) \Phi(F) u(\psi)^*, \forall \psi \in [\mathcal R].$$
By \cite[Theorem 10]{connes-feldman-weiss}, a countable nonsingular equivalence relation $\mathcal R$ is amenable if and only if it is hyperfinite. We will say that a countable nonsingular equivalence relation $\mathcal R$ is {\em nowhere amenable} if for every measurable subset $\mathcal U \subset X$ such that $\mu(\mathcal U) > 0$, the equivalence relation $\mathcal R | \mathcal U = \mathcal R \cap (\mathcal U \times \mathcal U)$ is nonamenable. 

Recall the following definition due to Gaboriau \cite[Definition IV.6]{Ga99}.

\begin{df}
Let $\mathcal R$ be a countable nonsingular equivalence relation on a standard measure space $(X, \mu)$ and $\mathcal R_1, \mathcal R_2 \subset \mathcal R$ subequivalence relations. We say that $\mathcal R$ {\em splits as the free product} $\mathcal R = \mathcal R_1 \ast \mathcal R_2$ if 
\begin{itemize}
\item $\mathcal R$ is generated by $\mathcal R_1$ and $\mathcal R_2$;
\item For every $p \in \N_{>0}$ and almost every $2p$-tuple $(x_j)_{j \in \Z/2p\Z}$ in $X$ such that $(x_{2i - 1}, x_{2i}) \in \mathcal R_1$ and $(x_{2i}, x_{2i + 1}) \in \mathcal R_2$ for all $i \in \Z/p\Z$, there exists $j \in \Z/2p\Z$ such that $x_j = x_{j + 1}$.
\end{itemize}
\end{df}

We have the following well-known fact:

\begin{prop}\label{free-equivalence}
Let $\mathcal R$ be a countable nonsingular equivalence relation on a standard measure space $(X, \mu)$ and $\mathcal R_1, \mathcal R_2 \subset \mathcal R$ subequivalence relations. Let $B = \rL^\infty(X)$, $M_1 = \rL(\mathcal R_1)$, $M_2 = \rL(\mathcal R_2)$, $M = \rL(\mathcal R)$ and denote by $E_1 : M_1 \to B$, $E_2 : M_2 \to B$, $E : M \to B$ the canonical faithful normal conditional expectations. The following conditions are equivalent:
\begin{enumerate}
\item $\mathcal R$ splits as the free product $\mathcal R = \mathcal R_1 \ast \mathcal R_2$.
\item $(M, E) = (M_1, E_1) \ast_B (M_2, E_2)$
\end{enumerate}
\end{prop}

We start by proving the following intermediate result in the framework of type ${\rm II_1}$ equivalence relations.

\begin{theo}\label{disintegration}
Let $\mathcal R$ be a countable (not necessarily ergodic) probability measure preserving  equivalence relation on a standard probability space $(X, \mu)$ which splits as a free product $\mathcal R = \mathcal R_1 \ast \mathcal R_2$ where $\mathcal R_i$ is a countable type ${\rm II_1}$ subequivalence relation for all $i \in \{1, 2\}$. 

Let $A \subset \rL(\mathcal R)$ be a Cartan subalgebra. Then $A \preceq_{\rL(\mathcal R)} \rL^\infty(X)$.
\end{theo}

\begin{proof}
Let $B = \rL^\infty(X)$, $M_1 = \rL(\mathcal R_1)$, $M_2 = \rL(\mathcal R_2)$ and $M = \rL(\mathcal R)$ so that $M = M_1 \ast_B M_2$. Let $A \subset M$ be a Cartan subalgebra. 

Assume first that both $\mathcal R_1$ and $\mathcal R_2$ are amenable and thus hyperfinite by \cite{connes-feldman-weiss}. Since both $\mathcal R_1$ and $\mathcal R_2$ are moreover of type ${\rm II_1}$, they are necessarily generated by a free pmp action of $\Z$. Hence $\mathcal R = \mathcal R_1 \ast \mathcal R_2$ is generated by a free pmp action of $\F_2$ and so $M \cong B \rtimes \F_2$. Then \cite[Theorem 1.6]{PV11} shows that $A \preceq_M B$.

Next assume that $\mathcal R_1$ or $\mathcal R_2$ is nonamenable. Without loss of generality, we may assume that $\mathcal R_1$ is nonamenable. Choose a measurable subset $\mathcal U \subset X$ such that $\mu(\mathcal U) > 0$ and $\mathcal R_1 | \mathcal U$ is nowhere amenable. Denote by $\mathcal V \subset X$ the $\mathcal R$-saturated measurable subset of $\mathcal U$ in $X$. Since $\mathcal R | \mathcal V = (\mathcal R_1 | \mathcal V) \ast (\mathcal R_2 | \mathcal V)$, we may assume that $\mu(\mathcal V) = 1$.

Since $\mathcal U$ is a complete section for $\mathcal R$, it follows from \cite[Th\'eor\`eme 44]{Al08} that we can write $\mathcal R | \mathcal U = \mathcal S_1 \ast \mathcal S_2$ where $\mathcal S_1 = \mathcal R_1 | \mathcal U$ and $\mathcal S_2$ is a type ${\rm II_1}$ subequivalence relation of $\mathcal R | \mathcal U$ which contains $\mathcal R_2 | \mathcal U$. 

Write $q = \mathbf 1_{\mathcal U} \in B$. By \cite[Corollary F.8]{BO08}, choose a projection $p \in A$ and a partial isometry $v \in M$ such that $v^*v = p$ and $vv^* = q$. Then $v A v^* \subset q M q$ is a Cartan subalgebra by \cite[Lemma 3.5]{popa-malleable1}. We can thus apply Theorem \ref{general-semifinite} to $\mathcal M = \rL(\mathcal S_1) \ast_{\rL^\infty(\mathcal U)} \rL(\mathcal S_2)$, $\mathcal A = v A v^*$ and $p = 1$. Then we obtain that $v A v^* \preceq_{qMq} Bq$, hence $A \preceq_M B$. 
\end{proof}

\begin{proof}[Proof of Theorem \ref{thmC}]
Write $B = \rL^\infty(X)$, $M_1 = \rL(\mathcal R_1)$, $M_2 = \rL(\mathcal R_2)$ and $M = \rL(\mathcal R)$ so that $M = M_1 \ast_B M_2$. Define on the standard infinite measure space $(X \times \R, m)$ the countable infinite measure preserving equivalence relations $\core(\mathcal R_1)$, $\core(\mathcal R_2)$ and $\core(\mathcal R)$ which are the Maharam extensions \cite{maharam} of the countable nonsingular equivalence relations $\mathcal R_1$, $\mathcal R_2$ and $\mathcal R$ respectively. Observe that both $\core(\mathcal R_1)$ and $\core(\mathcal R_2)$ are of type ${\rm II}$ and $\core(\mathcal R) = \core(\mathcal R_1) \ast \core(\mathcal R_2)$.

If we moreover write $\core(B) = \rL^\infty(X \times \R)$, we canonically have
$$\core(M_1) = \rL(\core(\mathcal R_1)), \; \core(M_2) = \rL(\core(\mathcal R_2)), \; \core(M) = \rL(\core(\mathcal R)) \; \text{ and } \; \core(M) = \core(M_1) \ast_{\core(B)} \core(M_2).$$

Let $A \subset M$ be a Cartan subalgebra. Using Notation \ref{notation}, we obtain that $\core(A) \subset \core(M)$ is a Cartan subalgebra. Let  $q \in \Proj(\core(B))$ such that $\Tr(q) = 1$. Up to cutting down by the central support of $q$ in $\core(M)$, we may assume that $q$ has central support equal to $1$ in $\core(M)$. By \cite[Lemma 2.1]{houdayer-vaes}, there exists $p \in \Proj(\core(A))$ and a partial isometry $v \in \core(M)$ such that $p = v^*v$ and $q = vv^*$. Observe that $v \core(A) v^* \subset q \core(M) q$ is still a Cartan subalgebra by \cite[Lemma 3.5]{popa-malleable1}. In order to show that $A$ and $B$ are unitarily conjugate inside $M$, using Theorem \ref{intertwining-general} and Proposition \ref{intertwining-core}, it suffices to show that $v \core(A) v^* \preceq_{\core(M)} \core(B)q$.

Let $\mathcal U \subset X \times \R$ be a measurable subset such that $\mathbf 1_{\mathcal U} = q$. Since $\mathbf 1_{\mathcal U}$ has central support equal to $1$ in $\core(M)$, $\mathcal U$ is a complete section for $\core(\mathcal R)$. By \cite[Th\'eor\`eme 44]{Al08}, we can write $\core(\mathcal R) | \mathcal U = \mathcal S_1 \ast \mathcal S_2$ where $\mathcal S_1 = \core(\mathcal R_1) | \mathcal U$ and $\mathcal S_2$ is a subequivalence relation of $\core(\mathcal R) | \mathcal U$ which contains $\core(\mathcal R_2) | \mathcal U$. In particular, both $\mathcal S_1$ and $\mathcal S_2$ are type ${\rm II_1}$ equivalence relations on the standard probability space $(\mathcal U, m | \mathcal U)$.

Let $\mathcal A = v \core(A) v^*$ and $\mathcal B = \rL^\infty(\mathcal U)$. Observe that $q\core(M)q = \rL(\core(\mathcal R) | \mathcal U) = \rL(\mathcal S_1 \ast \mathcal S_2)$ and $\mathcal A$ is a Cartan subalgebra in $\rL(\mathcal S_1 \ast \mathcal S_2)$. Then Theorem \ref{disintegration} implies that $\mathcal A \preceq_{\rL(\mathcal S_1 \ast \mathcal S_2)} \rL^\infty(\mathcal U)$, that is, $v \core(A) v^* \preceq_{\core(M)} \core(B)q$. This finishes the proof of Theorem \ref{thmC}.
\end{proof}

\section{Proof of Theorem \ref{thmD}}

We start by proving Theorem \ref{thmD} in the infinite measure preserving case. More precisely, we deduce the following result from its finite measure preserving counterpart proven in \cite[Theorem 1.1]{Io12a}.

\begin{theo}\label{gms-semifinite}
Let $\Gamma = \Gamma_1 \ast_{\Sigma} \Gamma_2$ be an amalgamated product group such that $\Sigma$ is finite and for all $i \in \{1, 2\}$, $\Gamma_i$ is infinite. Let $(\mathcal B, \Tr)$ be a type ${\rm I}$ von Neumann algebra endowed with a semifinite faithful normal trace. Let $\Gamma \curvearrowright (\mathcal B, \Tr)$ be a trace preserving action such that for all $i \in \{1, 2\}$, the crossed product von Neumann algebra $\mathcal B \rtimes \Gamma_i$ is of type ${\rm II}$. Put $\mathcal M = \mathcal B \rtimes \Gamma$. Let $p \in \Proj(\mathcal B)$ and $\mathcal A \subset p\mathcal M p$ any regular amenable von Neumann subalgebra.

Then for every nonzero projection $e \in \mathcal A ' \cap p \mathcal M p$, we have $\mathcal A e \preceq_{p\mathcal Mp} p \mathcal B p$.
\end{theo}

\begin{proof}
For every subset $\mathcal F \subset \Gamma$, denote by $P_{\mathcal F}$ the orthogonal projection from $\rL^2(\mathcal M, \Tr)$ onto the closed linear span of $\{x u_g  : x \in \mathcal B \cap \rL^2(\mathcal B, \Tr), g \in \mathcal F\}$. Since $\mathcal N_{p \mathcal M p}(\mathcal A)\dpr = p \mathcal M p$, Proposition \ref{intertwining-semifinite} (see also \cite[Lemma 2.7]{houdayer-vaes}) provides a central projection $z \in \mathcal Z(p \mathcal M p)$ and a net of unitaries $w_k \in \mathcal U(\mathcal A z)$ such that:
\begin{itemize}
\item $\lim_k \|P_{\mathcal F}(w_k)\|_{2, \Tr} = 0$ for all finite subset $\mathcal F \subset \Gamma$.
\item For every $\varepsilon > 0$, there exists a finite subset $\mathcal F \subset \Gamma$ such that $\|a - P_{\mathcal F}(a)\|_{2, \Tr} \leq \varepsilon$ for all $a \in \Ball(\mathcal A(p - z))$.
\end{itemize}

We prove by contradiction that $z = 0$. So, assume that $z \neq 0$. Recall that $\Gamma = \Gamma_{1} \ast_{\Sigma} \Gamma_{2}$. Hence the subgroup $\Sigma_{0} = \bigcap_{g \in \Gamma} g \Sigma g^{-1} < \Sigma$ is finite and normal in $\Gamma$. Define the quotient homomorphism $\rho : \Gamma \to \Gamma / \Sigma_{0}$ and put $\Lambda = \Gamma / \Sigma_0$, $\Lambda_i = \Gamma_i / \Sigma_0$ for $i \in \{1, 2\}$, $\Upsilon = \Sigma / \Sigma_{0}$ so that $\Lambda = \Lambda_{1} \ast_{\Upsilon} \Lambda_{2}$. We get that $\bigcap_{s \in \Lambda} s \Upsilon s^{-1} = \{e\}$, hence $\rL(\Lambda)$ is a ${\rm II_1}$ factor which does not have property Gamma by \cite[Corollary 6.2]{Io12a}.

Define the unitary $W \in \mathcal U(\rL^2(\mathcal B, \Tr) \otimes \ell^2(\Gamma) \otimes \ell^2(\Lambda))$ by 
$$W(\xi \otimes \delta_g \otimes \delta_s) = \xi \otimes \delta_g \otimes \delta_{\rho(g^{-1})s}, \forall \xi \in \rL^2(\mathcal B, \Tr), \forall g \in \Gamma, \forall s \in \Lambda.$$
Next, define the {\em dual coaction} $\Delta_\rho : \mathcal M \to \mathcal M \ovt \rL(\Lambda)$ by $\Delta_\rho(x) = W^*(x \otimes 1)W$ for all $x \in \mathcal M$. Observe that $\Delta_\rho$ is a trace preserving $\ast$-embedding which satisfies $\Delta_\rho(b u_g) = b u_g \otimes v_{\rho(g)}$ for all $b \in \mathcal B$ and all $g \in \Gamma$.

For every subset $\mathcal F \subset \Gamma$, denote by $Q_{\rho(\mathcal F)}$ the orthogonal projection from $\rL^2(\rL(\Lambda))$ onto the closed linear span of $\{v_{\rho(g)} : g \in \mathcal F\}$. Observe that $(1 \otimes Q_{\rho(\mathcal F)})(\Delta_\rho(x)) = \Delta_\rho(P_{\Sigma_0 \mathcal F}(x))$ for all $x \in \mathcal M$. Since $\Delta_\rho$ is $\|\cdot\|_{2, \Tr}$-preserving and since $\Sigma_0$ is finite, for any finite subset $\mathcal F \subset \Gamma$, we have
$$\lim_k \|(1 \otimes Q_{\rho(\mathcal F)})(\Delta_\rho(w_k)) \|_2 = \lim_k \|\Delta_\rho(P_{\Sigma_0\mathcal F}(w_k)) \|_2 = 0.$$
Since $\Upsilon < \Lambda$ is a finite subgroup, this implies that $\Delta_\rho(\mathcal A z) \npreceq_{\mathcal M \ovt \rL(\Lambda)} q\mathcal M q \ovt \rL(\Upsilon)$ for all $q \in \Proj(\mathcal B)$.

Put $\widetilde \Lambda = \Lambda \ast_{\Upsilon} (\Upsilon \times \F_2) = \Lambda_1 \ast_{\Upsilon} \Lambda_2 \ast_{\Upsilon} (\Upsilon \times \F_2)$ and consider the malleable deformation $(\alpha_t)$ on $\rL(\widetilde \Lambda)$ from Section \ref{section-malleable}. Define $N < \Lambda$ the normal subgroup generated by  $\{t \Lambda t^{-1} : t \in \F_2\}$ so that $\rL(\widetilde \Lambda) = \mathcal N \rtimes \F_2$ with $\mathcal N = \rL(N)$. Applying Popa-Vaes's dichotomy result \cite[Theorem 1.6]{PV11} to each of the inclusions 
 $$(\id \otimes \alpha_t)(\Delta_\rho(\mathcal A z)) \subset p \mathcal M p \ovt \rL(\widetilde \Lambda) = (p \mathcal M p \ovt \mathcal N) \rtimes \F_2 \; \mbox{ with } \; t \in (0, 1),$$ we obtain that at least one of the following holds true:
\begin{enumerate}
\item Either there exists $t \in (0, 1)$ such that $(\id \otimes \alpha_t)(\Delta_\rho(\mathcal A z)) \preceq_{p\mathcal Mp \ovt \rL(\widetilde \Lambda)} p\mathcal Mp \ovt \mathcal N$.
\item Or for all $t \in (0, 1)$, $(\id \otimes \alpha_t)(\Delta_\rho(p\mathcal Mp))$ is amenable relative to $p \mathcal M p \ovt \mathcal N$ inside $p \mathcal M p \ovt \rL(\widetilde \Lambda)$.
\end{enumerate}
We will prove below that each case leads to a contradiction.

In case $(1)$, by  \cite[Theorem 3.2]{Io12a} and since $\Delta_\rho(\mathcal A z) \npreceq_{p \mathcal M p \ovt \rL(\Lambda)} p \mathcal M p \ovt \rL(\Upsilon)$ and $\mathcal N_{p\mathcal M p z}(\mathcal A z)\dpr = p\mathcal Mp z$, there exists $i \in \{ 1, 2 \}$ such that $\Delta_\rho(p\mathcal Mp z) \preceq_{p\mathcal Mp \ovt \rL(\Lambda)} p\mathcal Mp \ovt \rL(\Lambda_i)$. In order to get a contradiction, we will need the following.

\begin{claim}
Let $e \in \Proj(\mathcal M)$, $\mathcal Q \subset e \mathcal M e$ any von Neumann subalgebra and $\mathcal S$ any nonempty collection of subgroups of $\Gamma$. If $\mathcal Q \npreceq_{ \mathcal M} q (\mathcal B \rtimes H) q$ for all $H \in \mathcal S$ and all $q \in \Proj(\mathcal B)$, then $\Delta_\rho(\mathcal Q) \npreceq_{ \mathcal M \ovt \rL(\Lambda)} q \mathcal M q \ovt \rL(\rho(H))$ for all $H \in \mathcal S$ and all $q \in \Proj(\mathcal B)$. 
\end{claim}

\begin{proof}[Proof of the Claim]
Since $\mathcal Q \npreceq_{ \mathcal M } q (\mathcal B \rtimes H) q$ for all $H \in \mathcal S$ and all $q \in \Proj(\mathcal B)$, Proposition \ref{intertwining-semifinite} implies that there exists a net $v_k \in \mathcal U(\mathcal Q)$ such that $\lim_k \|P_{\mathcal F}(v_k)\|_{2, \Tr} = 0$ for all subsets $\mathcal F \subset \Gamma$ which are small relative to $\mathcal S$. Observe that since $\Sigma_0$ is finite, $\Sigma_0 \mathcal F$ is small relative to $\mathcal S$ for all subsets $\mathcal F \subset \Gamma$ which are small relative to $\mathcal S$. Moreover, $(1 \otimes Q_{\rho(\mathcal F)})(\Delta_\rho(x)) = \Delta_\rho(P_{\Sigma_0 \mathcal F}(x))$ for all $x \in \mathcal Q$ and all subsets $\mathcal F \subset \Gamma$ which are small relative to $\mathcal S$. Since $\Delta_\rho$ is $\|\cdot\|_{2, \Tr}$-preserving, for all subsets $\mathcal F \subset \Gamma$ which are small relative to $\mathcal S$, we have
\begin{equation}\label{convergence}
\lim_k \|(1 \otimes Q_{\rho(\mathcal F)})(\Delta_\rho(v_k)) \|_{2, \Tr} = \lim_k \|\Delta_\rho(P_{\Sigma_0 \mathcal F}(v_k)) \|_{2, \Tr} = 0.
\end{equation}

Denote by $\rho(\mathcal S)$ the nonempty collection of subgroups $\rho(H) \subset \Lambda$ with $H \in \mathcal S$. Let $\mathcal G \subset \Lambda$ be any subset which is small relative to $\rho(\mathcal S)$. Then there exist $n \geq 1$, $H_1, \dots, H_n \in \mathcal S$ and $s_1, t_1, \dots, s_n, t_n \in \Lambda$ such that $\mathcal G \subset \bigcup_{i = 1}^n s_i \rho(H_i) t_i$. Choose $g_i, h_i \in \Gamma$ such that $\rho(g_i) = s_i$ and $\rho(h_i) = t_i$ and denote $\mathcal F = \bigcup_{i = 1}^n g_i H_i h_i$. Then $\mathcal G \subset \rho(\mathcal F)$. Therefore, $(\ref{convergence})$ implies that $\lim_k \|(1 \otimes Q_{\mathcal G})(\Delta_\rho(v_k)) \|_{2, \Tr} = 0$ for all subsets $\mathcal G \subset \Lambda$ which are small relative to $\rho(\mathcal S)$. Thus, Proposition \ref{intertwining-semifinite} implies that $\Delta_\rho(\mathcal Q) \npreceq_{\mathcal M \ovt \rL(\Lambda)} q \mathcal M q \ovt \rL(\rho(H))$ for all $H \in \mathcal S$ and all $q \in \Proj(\mathcal B)$. 
\end{proof}

We apply the Claim to $\mathcal Q = p\mathcal Mp z$ and $\mathcal S = \{\Gamma_1, \Gamma_2\}$. In order to do that, we need to check that $p\mathcal Mp z \npreceq_{q \mathcal M q} q(\mathcal B \rtimes \Gamma_i)q$ for all $i \in \{1, 2\}$ and all $q \in \Proj(\mathcal B)$. Since $\mathcal B \rtimes \Sigma$ is a type ${\rm I}$ von Neumann algebra and $\mathcal B \rtimes \Gamma_i$ is a type ${\rm II}$ von Neumann algebra, Proposition \ref{intertwining-afp} yields the result. Therefore, by the Claim, we get that $\Delta_\rho(p\mathcal Mp z) \npreceq_{p\mathcal Mp \ovt \rL(\Lambda)} p\mathcal Mp \ovt \rL(\Lambda_i)$ for all $i \in \{1, 2\}$. This is a contradiction. 

In case $(2)$, since $\rL(\Lambda)$ does not have property Gamma, \cite[Theorem 5.2]{Io12a} shows that either there exists $i \in \{1, 2\}$ such that $\rL(\Lambda) \preceq_{\rL(\Lambda)} \rL(\Lambda_i)$ or $\rL(\Lambda)$ is amenable. Both of these cases are easily seen to lead to a contradiction. This finishes the proof of Theorem \ref{gms-semifinite}.
\end{proof}

\begin{proof}[Proof of Theorem \ref{thmD}]
Let now $\Gamma \curvearrowright (X, \mu)$ be any nonsingular free ergodic action on a standard measure space such that for all $i \in \{1, 2\}$, the restricted action $\Gamma_i \curvearrowright (X, \mu)$ is recurrent. Let $B = \rL^\infty(X)$ and put $M = B \rtimes \Gamma$. Assume that $A \subset M$ is another  Cartan subalgebra. 

Since $A, B \subset M$ are both tracial von Neumann subalgebras of $M$ with expectation, we use Notation \ref{notation}. Define $\core(B) = \rL^\infty(X \times \R)$ and consider the Maharam extension $\Gamma \curvearrowright \core(B)$ of the action $\Gamma \curvearrowright B$ so that we canonically have $\core(M) = \core(B) \rtimes \Gamma$. Observe that for all $i \in \{1, 2\}$, the action $\Gamma_i \curvearrowright \core(B)$ is still recurrent so that $\core(B) \rtimes \Gamma_i$ is a  type ${\rm II}$ von Neumann algebra.

Let $p \in \Proj(\core(A))$. By \cite[Lemma 2.1]{houdayer-vaes}, there exist $q \in \Proj(\core(B))$ and a partial isometry $v \in \core(M)$ such that $p = v^*v$ and $q = vv^*$. Observe that $v \core(A) v^* \subset q \core(M) q$ is still a Cartan subalgebra by \cite[Lemma 3.5]{popa-malleable1}. 

By Theorem \ref{gms-semifinite}, we get $v \core(A) v^* \preceq_{q \core(M) q} \core(B) q$. By Proposition \ref{intertwining-core}, this implies that $A \preceq_M B$. Since $M$ is a factor, by \cite[Theorem 2.5]{houdayer-vaes}, we get that there exists a unitary $u \in \mathcal U(M)$ such that $uAu^* = B$. This finishes the proof of Theorem \ref{thmD}.
\end{proof}

\section{AFP von Neumann algebras with many nonconjugate Cartan subalgebras}\label{non-uniqueness}

Connes and Jones exhibited in \cite{CJ81} the first examples  of  ${\rm II_1}$ factors $M$ with at least two Cartan subalgebras which are not conjugate by an automorphism of $M$. More concrete examples were found by Ozawa and Popa in \cite{OP08}.

Recently, Speelman and Vaes exhibited in \cite{SV11} the first examples of group measure space ${\rm II_1}$ factors $M = \rL^\infty(Y) \rtimes \Lambda$ with uncountably many non stably conjugate Cartan subalgebras. Recall from \cite{SV11} that two Cartan subalgebras $A$ and $B$ of a ${\rm II_1}$ factor $N$ are {\em stably conjugate} if there exists nonzero projections $p \in A$ and $q \in B$ and a surjective $\ast$-isomorphism $\alpha : pNp \to qNq$ such that $\alpha(Ap) = Bq$. Put $\mathcal N = N \ovt \mathbf B(\ell^2)$, $\mathcal A = A \ovt \ell^\infty$ and $\mathcal B = B \ovt \ell^\infty$. Observe that $\mathcal A$ and $\mathcal B$ are Cartan subalgebras in the type ${\rm II_\infty}$ factor $\mathcal N$. Moreover, we have that $A$ and $B$ are stably conjugate in $N$ if and only if $\mathcal A$ and $\mathcal B$ are conjugate in $\mathcal N$.

Let $\Lambda \curvearrowright (Y, \nu)$ be a probability measure preserving free ergodic action  as in the statement of \cite[Theorem 2]{SV11} so that the corresponding group measure space ${\rm II_1}$ factor $N = \rL^\infty(Y) \rtimes \Lambda$ has uncountably many non stably conjugate Cartan subalgebras.

Put $\Gamma = \Lambda \ast \Z$ and consider the \emph{induced} action $\Gamma \curvearrowright (X, \mu)$ with $X = \Ind_\Lambda^\Gamma Y$. Observe that $\Gamma \curvearrowright (X, \mu)$ is an infinite measure preserving free ergodic action. Write $\mathcal M = \rL^\infty(X) \rtimes \Gamma$ for the corresponding group measure space type ${\rm II_\infty}$ factor. Since $\Gamma = \Lambda \ast \Z$, we canonically have $\mathcal M = \mathcal M_1 \ast_{\mathcal B} \mathcal M_2$ with $\mathcal B = \rL^\infty(X)$, $\mathcal M_1 = \mathcal B \rtimes \Lambda$ and $\mathcal M_2 = \mathcal B \rtimes \Z$. On the other hand, we also have
$$\mathcal M = (\rL^\infty(Y) \rtimes \Lambda) \ovt \mathbf B(\ell^2(\Gamma / \Lambda)) = N \ovt \mathbf B(\ell^2(\Gamma / \Lambda)).$$

Therefore we obtain the following result.

\begin{theo}
The amalgamated free product type ${\rm II_\infty}$ factor $\mathcal M = \mathcal M_1 \ast_{\mathcal B} \mathcal M_2$ has uncountably many nonconjugate Cartan subalgebras.
\end{theo}

This result shows that the condition in Theorem \ref{thmD} imposing recurrence of the action $\Gamma_i \curvearrowright (X, \mu)$ for all $i \in \{1, 2\}$, is indeed necessary.

\end{document}